\documentclass{article}

\usepackage{graphicx}
\usepackage{tikz}
\usepackage{tikz-cd}
\usepackage{amsmath}
\usepackage{amsfonts}
\usepackage{amsthm}
\usepackage{amssymb}
\usepackage{tabularx}
\usepackage{booktabs}
\usepackage{url}

\newcommand{\Z}{\mathbb Z}
\newcommand{\Q}{\mathbb Q}
\newcommand{\R}{\mathbb R}

\DeclareMathOperator{\Spec}{Spec}
\DeclareMathOperator{\Pic}{Pic}

\newtheorem{definition}{Definition}
\newtheorem{corollary}{Corollary}
\newtheorem{proposition}{Proposition}
\newtheorem{theorem}{Theorem}
\newtheorem{lemma}{Lemma}
\newtheorem{porism}{Porism}

\title{A scheme-theoretic interpretation of generalized splines}

\author{Kyle Stoltz}

\date{} 

\begin{document}

\maketitle

\begin{abstract}
Classical splines feature prominently in approximation theory and numerical analysis, while GKM theory arises in the study of equivariant cohomology. More recently, generalized splines have been studied which simultaneously generalize both classical splines and GKM theory. We show that generalized splines can be understood as a scheme-theoretic phenomenon, and establish the foundations for working with generalized splines via scheme theory. To do this, we first prove that the ring of generalized splines over an edge-labeled graph is isomorphic to the limit of a diagram associated to the edge-labeled graph. This is used to establish the connection to scheme theory, and in particular to secure a local--global principle for generalized splines.

The secondary aim of the paper is to catalog a variety of initial local--global results. We prove a module of multivariate generalized splines over $k[x,y]$ is free when an edge-labeled graph is locally trivial or determined by a cycle on a standard cover, provide criteria for generalized splines over UFDs to be finite projective, for multivariate generalized splines to be free, along with other results. The appendix includes a variety of examples which demonstrate precise experimental control over when a module of splines is free, with computer verification by Macaulay2. We conclude with an account of how deletion and contraction affect the spectrum of a ring of splines.
\end{abstract}

\section{Introduction}
Classical splines are one of the major objects of study within numerical analysis, where they allow for piecewise approximation of a function with prescribed smoothness. As an elementary example, between two consecutive $x$-intercepts of $y=\sin(x)$, the sine curve resembles a parabola. Thus, we could imagine constructing a piecewise function whose components consist of various parabolas such that adjacent ones agree along their neighbors' $x$-intercepts. A better fit would be between cubic polynomials that not only agree at $x$-intercepts, but which have one or more derivatives which agree at $x$-intercepts. Modern examples include higher dimensional domains in addition to more complex smoothness conditions.

More recently, generalized splines have been developed which allow a similar formulation for arbitrary commutative rings, which utilize an edge-labeled graph $(G, \alpha)$ with edges labeled by ideals of a commutative ring $R$. Generalized splines, then, show up as vertex labels such that adjacent vertex labels agree modulo the ideal of their connecting edge. As a basic example, let $G$ be a graph with two vertices and one edge and let $\alpha : G \rightarrow \text{Ideals}(\R[x])$ be an edge labeling which labels the only edge $\langle x-1 \rangle$. Here $(G, \alpha)$ forms an edge-labeled graph. A generalized or ``graph'' spline then is a vertex labeling by two polynomials $P(x), Q(x) \in \R[x]$ such that both agree modulo $\langle x - 1 \rangle$, which just says $P(1) = Q(1)$. So, we can see generalized splines capture agreement of functions at a prescribed point. More generally, however, we could have selected $G$ and let the edge be labeled by the integer ideal $\langle 3 \rangle$. Then a generalized spline for the edge-labeled graph would be a selection of two integers $p,q \in \Z$ such that $p$ and $q$ agree modulo $\langle 3 \rangle$.

Scheme theory provides an algebra-geometry dictionary, where every commutative ring $R$ has an associated locally ringed space $\Spec(R)$, where $R$ is recovered as the global sections of $\Spec(R)$. Here $\Spec(R)$ is the set of prime ideals of $R$ equipped with a topology generated by (complements of) closed sets $V(r) = \{r \equiv 0 \pmod{\mathfrak p} : \mathfrak p \in \Spec(R)\}$ for each $r \in R$, and a structure sheaf. This topology is justified by analogy with the remainder theorem from elementary algebra: polynomial evaluation is equivalent to identifying remainders modulo a linear polynomial. That is, instead of algebraic substitution to evaluate a polynomial at $x=a$, we can determine its remainder when divided by the linear polynomial $x-a$. Thus, we can think of $P(x) \equiv Q(x) \pmod{ x - 1}$ as being in direct analogy with $p \equiv q \pmod{3}$ as generalized function substitutions. The topology for schemes then is a generalization of the Zariski topology for algebraic varieties where closed sets are determined by the zero loci of polynomials.

Since generalized splines were explicitly designed to generalize agreement of functions along a specified locus, the previous remark is unsurprising. However, the scheme-theoretic implication suggests that generalized splines capture the intuitive notion of classical splines as piecewise functions in a more literal geometric sense than has been noted in the literature for arbitrary commutative rings. That is, while generalized splines were made precisely to generalize classical splines, the generalization has not been considered as a scheme-theoretic phenomenon.

The purpose of this paper is twofold: the first is to provide a systematic account of the intrinsic relationship between generalized splines and schemes. The second purpose is to catalog an initial collection of results which utilize the new perspective, so that its use is clear. In particular, local-global arguments are ubiquitous across mathematics, so their introduction into the study of generalized splines unlocks a variety of related techniques that were previously inaccessible. The local-global applications we collect in this paper exemplify how such techniques can be used to derive concrete results about generalized splines but are far from exhaustive. An appendix is included which features a variety of examples where the results are used to control freeness of a module of splines, in addition to computer verification with Macaulay2. The appendix demonstrates that the results provide precise experimental control over when a module of splines is free.

A final remark is in order: schemes are notoriously difficult to learn and use effectively. However, the author is also of the opinion that the amount of scheme theory required is fairly low, so anyone who wishes to study the material will not need to learn schemes entirely. As an analogy with computer programming, it is not needed to learn all of $C++$ to write code, and in our case a meager subset (think, just $C$) suffices.

\section{Algebraic preliminaries}
Before we can move into scheme theory, we need to first establish a few algebraic preliminaries. To ease notation, we will write edge-labeled graphs with $G$ in lieu of the usual $(G, \alpha)$, with the edge-label implied. We will also assume all edge-labeled graphs $G$ are finite, and therefore we will also assume that every module of generalized splines is a finite $R$-module. Thus, some results which require finite as a hypothesis will not explicitly invoke finiteness, since it automatically holds in situations we will consider.

Let $R$ be a commutative ring, and let $G$ be an edge-labeled graph with edge labeling $\alpha : E \rightarrow \{\text{Ideals of } R\}$. Let $s$ be a graph spline, i.e. a function $s : V \rightarrow R$ such that $s(u) - s(v) \in \alpha(e_{uv})$, where $u,v \in V$ and $e_{uv} \in E$ is the edge that connects $u$ and $v$. Notice the GKM-condition is equivalent to $s(u) \equiv s(v) \bmod{\alpha(e_{uv})}$ and that $(s(u), s(v))$ is in the fiber product of the diagram $R \rightarrow R/\alpha(e_{uv}) \leftarrow R$. These observations form the basis for spline diagrams, which are established as follows.

\begin{definition}
Let $R$ be a commutative ring, and $G$ an edge-labeled graph. Then there is an associated \textbf{spline diagram} in the category of commutative rings, denoted $\mathfrak G$, formed as follows: for each $u,v \in V$ and edge $e_{uv} \in E$, the rings $R_u = R$, $R_v = R$, and $R/\alpha(e_{uv})$, with morphisms $R_u \rightarrow R/\alpha(e_{uv}) \leftarrow R_v$.

The limit of $\mathfrak G$ is the \textbf{ring of limit splines}. Any diagram in the category of commutative rings that can be obtained by replacing the $R/I$ and associated edge morphisms via the isomorphism theorems for rings is also referred to as a spline diagram. Any edge-labeled graph with ideals satisfying the isomorphism requirements specified is said to be an \textbf{associated} edge-labeled graph.
\end{definition}

As an example, consider the edge-labeled graph with three vertices in Figure~\ref{fig:edge-triangle}. This is associated with the spline diagram in Figure~\ref{fig:triangle-spline-diagram}. Because the same ring $R$ appears as an object in several locations in the diagram, each copy is referred to by the associated vertex, $R_u=R$, $R_v=R$, and $R_w=R$. For any commutative ring $R$, and a selection of edge ideals $I_{uv}$, $I_{vw}$, and $I_{wu}$, the ring of splines for the edge-labeled graph is isomorphic to the limit of this spline diagram. The situation is pictured in Figure~\ref{fig:triangle-limit}, noting that each projection is named by the associated vertices.

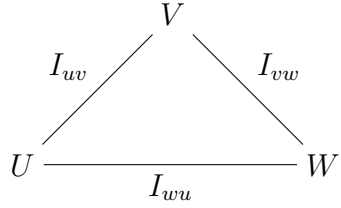
\begin{figure}
\centering
\begin{tikzpicture}[>=stealth, every node/.style={font=\large}]
  \node (A) at (0,0) {$U$};
  \node (B) at (2,2) {$V$};
  \node (C) at (4,0) {$W$};
  \draw (A) -- node[above left, midway] {$I_{uv}$} (B);
  \draw (B) -- node[above right, midway] {$I_{vw}$} (C);
  \draw (C) -- node[below, midway] {$I_{wu}$} (A);
\end{tikzpicture}
\caption{An edge-labeled graph with ideals $I_{uv}, I_{vw}, I_{wu}$.}
\label{fig:edge-triangle}
\end{figure}

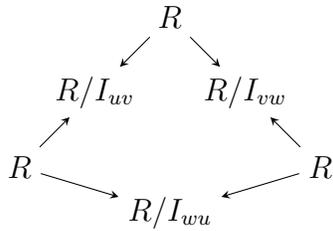
\begin{figure}
\centering
\begin{tikzpicture}[>=stealth, every node/.style={font=\large}]
  \node (A) at (0,0) {$R$};
  \node (B) at (2,2) {$R$};
  \node (C) at (4,0) {$R$};
  \node (QAB) at (1,1)  {$R/I_{uv}$};
  \node (QBC) at (3,1)  {$R/I_{vw}$};
  \node (QCA) at (2,-0.6) {$R/I_{wu}$};
  \draw[->] (A) -- (QAB);
  \draw[->] (B) -- (QAB);
  \draw[->] (B) -- (QBC);
  \draw[->] (C) -- (QBC);
  \draw[->] (C) -- (QCA);
  \draw[->] (A) -- (QCA);
\end{tikzpicture}
\caption{Spline diagram of rings associated to Figure~\ref{fig:edge-triangle}.}
\label{fig:triangle-spline-diagram}
\end{figure}

\begin{figure}
\centering
\begin{tikzpicture}[>=stealth, every node/.style={font=\large}]
  \node (A) at (0,0) {$R$};
  \node (B) at (2,2) {$R$};
  \node (C) at (4,0) {$R$};
  \node (QAB) at (1,1) {$R/I_{uv}$};
  \node (QBC) at (3,1) {$R/I_{vw}$};
  \node (QCA) at (2,-0.6) {$R/I_{wu}$};
  \node[draw, rounded corners, inner sep=4pt] (L) at (1,7) {$P$};
  \draw[->] (A) -- (QAB);
  \draw[->] (B) -- (QAB);
  \draw[->] (B) -- (QBC);
  \draw[->] (C) -- (QBC);
  \draw[->] (C) -- (QCA);
  \draw[->] (A) -- (QCA);
  \draw[->] (L) -- node[above left=-2pt] {$\pi_u$}  (A);
  \draw[->] (L) -- node[above]           {$\pi_v$}  (B);
  \draw[->] (L) -- node[above right=-2pt]{$\pi_w$}  (C);
  \draw[->] (L) -- node[left]  {$\pi_{uv}$} (QAB);
  \draw[->] (L) -- node[right] {$\pi_{vw}$} (QBC);
  \draw[->] (L) -- node[below] {$\pi_{wu}$} (QCA);
\end{tikzpicture}
\caption{Limit object $P$ with projections to the vertex and edge objects.}
\label{fig:triangle-limit}
\end{figure}
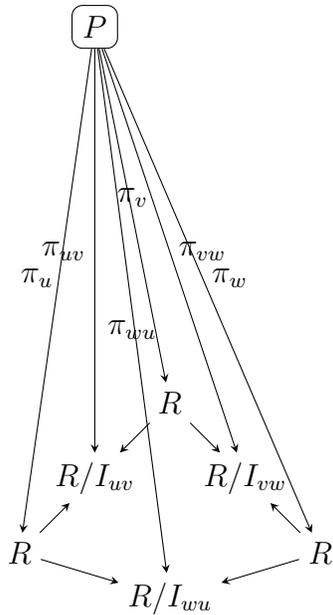

It was brought to the author's attention by Michael DiPasquale that this is related to a construction in \textit{modules of piecewise polynomials and their freeness} by Billera and Rose, who constructed classical spline modules as inverse limits \cite{billera1992modules}. Thus, obtaining generalized splines as a limit of a diagram is a natural extension of the fact classical splines can be obtained as a limit of a diagram. That is, we have effectively observed that Billera and Rose's classical spline construction manifests in the generalized spline context. Continuing, the following theorem justifies the limit claim.

\begin{theorem}
\label{thm:ring-limit-splines}
The ring of splines $R_{G}$ over the edge-labeled graph $G$ is isomorphic to the associated ring of limit splines.
\end{theorem}
\begin{proof}
Note that $R_{G}$ can be equipped with projections $\pi_u^G$ to the objects $R_u = R$ and $R_{uv} = R/I_{uv}$ in the diagram via $\pi_u^G(s) = s(u)$. That is, each spline $s$ maps to the label it assigns to the vertex $u$, as an element of $R_u$. Also let $\pi_{uv}^G(s):=s(u)+I_{uv}=s(v)+I_{uv}\in R/I_{uv}$, which is well defined by the GKM-condition. Let $P$ be the ring of limit splines with projections $\pi_u^P : P \rightarrow R_u$ and $\pi^P_{uv} : P \rightarrow R/I_{uv}$. We will produce an isomorphism $f: P \rightarrow R_{G}$ which commutes with the projections.

Let $f: P \rightarrow R_{G}$ be given by mapping $p \in P$ to the spline $s \in R_{G}$ which labels the vertex $u$ as $s(u) = \pi_u^P(p)$. That is, $s$ labels each vertex $u$ with the image of $p$ under the projection $\pi_u^P$. We know such an $s$ is a spline because the projections commute with every edge diagram $R_u \rightarrow R/I_{uv} \leftarrow R_v$, which implies $s(u) \equiv s(v) \bmod I_{uv}$ for all $u$ and $v$ connected by an edge. Now notice:

\begin{align*}
  \pi_u^G(f(p)) &= \pi_u^G(s) \\
                &= s(u) \\
                &= \pi_u^P(p)
\end{align*}

which shows $f$ commutes with the projections. To see that $f$ is one-to-one, suppose $p_1, p_2 \in P$ satisfy $f(p_1) = f(p_2) = s$. Then by definition of $f$ we have that $\pi_u^P(p_1) = s(u)$ and also $\pi_u^P(p_2) = s(u)$ for all projections. But elements of $P$ are identified by their images under the projections, hence $p_1=p_2$ and $f$ is injective.

To see $f$ is surjective, let $\phi$ be the unique morphism from $R_{G}$ to $P$ given by the limit property of $P$. Notice each $s$ maps to $p = \phi(s)$ such that $s(u) = \pi_u^P(p)$ for each $u$. But then by definition of $f$, we have that $f(p) = s$. Hence $f$ is also surjective. Since we already showed $f$ is injective, it follows $f$ is an isomorphism and hence $R_{G} \cong P$ as claimed.
\end{proof}

So we have established that for a ring $R$ and edge-labeled graph $G$, there is a spline diagram $\mathfrak G$ such that its ring of limit splines is isomorphic to the associated ring of generalized splines. It is not the case that a given spline diagram is uniquely associated to an edge-labeled graph. This is because the definition of spline diagram as given allows replacing the edge objects $R/I$ with isomorphic copies, and it is possible for a ring to have two ideals, $I$ and $J$, such that $R/I \cong R/J$. However, since we are studying the rings of splines themselves, we can associate a spline diagram non-canonically with an edge-labeled graph, and then study the ring of splines.

Our next goal is to show that the ring of limit splines can be obtained after finitely many pullbacks or equalizers of surjections. The motivation for why this is needed is intuitive from the view of schemes, but opaque algebraically. Thus, it may be instructive to skip to Section~\ref{sec:pushouts} now, and return to this section later once the motivation is clear.

\begin{theorem}
\label{thm:iterated-pullback}
Let $\mathfrak G$ be a spline diagram with ring of limit splines $R_{G}$, and assume $G$ is connected. Then $R_{G}$ is isomorphic to a ring obtained after finitely many pullbacks or equalizers of surjections. If $G$ is not connected, then $R_{G}$ is obtained as a product of $R_{G'}$ with each $G'$ a connected component of $G$ and the $R_{G'}$ obtained as in the connected case.
\end{theorem}
\begin{proof}
The fact $R_{G}$ is obtained as a product along connected components of $G$ is a basic result in the generalized spline literature, so we only need to prove the claim about the connected components. We proceed inductively. Let $n$ be the number of edges. For $n=1$ we have the spline diagram $R_u \rightarrow R/I_{uv} \leftarrow R_v$. Since the category of commutative rings has all small limits we know the fiber product over the spline diagram exists and is clearly isomorphic to the limit of the diagram, so the claim holds when $n=1$. For $n=2$, we can pullback over each edge diagram, and then pullback to a second level fiber product as in Figure~\ref{fig:second-level-pullback}.

\begin{figure}
\centering
\begin{tikzcd}[row sep=4em, column sep=1.2em]
  &[-0.4em]
  {} &
  \makebox[0pt]{\(\bigl(R_u \times_{R/I_{uv}} R_v\bigr)\!\times_{R_v}\!\bigl(R_v \times_{R/I_{vw}} R_w\bigr)\)}
  \arrow[dl,"\pi_1"', shorten <=2pt, shorten >=2pt]
  \arrow[dr,"\pi_2", shorten <=2pt, shorten >=2pt]
  &[-0.4em] {} \\[1.0em]
  &
  R_u \times_{R/I_{uv}} R_v
  \arrow[dl,"\pi_{u}"']
  \arrow[dr,"\pi_{v1}"]
  &[-0.6em]
  {}
  &[-0.6em]
  R_v \times_{R/I_{vw}} R_w
  \arrow[dl,"\pi_{v2}"']
  \arrow[dr,"\pi_{w}"]
  & \\[0.3em]
  R_u
  \arrow[r]
  & R/I_{uv}
  &[-0.6em]
  R_v
  \arrow[l] \arrow[r]
  &[-0.6em]
  R/I_{vw}
  & R_w
  \arrow[l]
\end{tikzcd}
\caption{Second-level pullback describing simultaneous GKM-conditions for two edges.}
\label{fig:second-level-pullback}
\end{figure}

Notice that elements of the second level fiber product $\bigl(R_u \times_{R/I_{uv}} R_v\bigr)\! \times_{R_v}\! \bigl(R_v \times_{R/I_{vw}} R_w\bigr)$ are of the form $((u,v_1), (v_2, w))$  with $u \in R_u$, $v_1, v_2 \in R_v$ and $w \in R_w$. Because the diagram commutes we can see that:
\begin{align*}
    (\pi_{v1} \circ \pi_1)\bigl((u,v_1), (v_2, w)\bigr)
      &= (\pi_{v2} \circ \pi_2)\bigl((u,v_1), (v_2, w)\bigr) \\
    \pi_{v1}(u,v_1) &= \pi_{v2}(v_2, w) \\
    v_1 &= v_2.
\end{align*}

Furthermore, we can see that the pairs $(u,v) \in R_u \times_{R/I_{uv}} R_v$ are precisely those satisfying the GKM-condition over the left edge diagram, while $(v, w) \in R_v \times_{R/I_{vw}} R_w$ are those satisfying the GKM-condition over the right edge diagram. Thus, the second level fiber product elements $((u, v_1), (v_2, w))$ are precisely the $u, v, w$ which simultaneously satisfy the GKM-condition for both the left edge diagram and right edge diagram. 

Thus, we obtain a morphism of rings $f: R_{G} \rightarrow \bigl(R_u \times_{R/I_{uv}} R_v\bigr)\! \times_{R_v}\! \bigl(R_v \times_{R/I_{vw}} R_w\bigr)$ by mapping $s \in R_{G}$ to $((s(u), s(v)), (s(v), s(w)))$. To see $f$ is an isomorphism, notice that since splines are uniquely identified by how they label vertices the map is one-to-one. Since the $((u, v_1), (v_2, w))$ satisfy the GKM-condition simultaneously, they imply the existence of a spline which maps to $((u, v_1), (v_2, w))$ under $f$, so $f$ is also surjective. Since $f$ is a bijective morphism of rings we have that:

\[
R_{G} \cong \bigl(R_u \times_{R/I_{uv}} R_v\bigr)\! \times_{R_v}\! \bigl(R_v \times_{R/I_{vw}} R_w\bigr).
\]

So we can see that for the base cases for edges $n=1$ and $n=2$, $R_{G}$ is isomorphic to a ring obtained by repeated pullback. We now need to show that each pullback is taken along surjections. It is clear the first level fiber products are pullbacks along surjections, since $R_{i} \rightarrow R/I_{ij}$ and $R_j \rightarrow R/I_{ij}$ are surjective. 

To see the second level fiber product is a pullback along surjections, we need to show $\pi_{v1}$ and $\pi_{v2}$ are surjective. That is, we need to show that for any $a \in R_v$, we can find $(a,a) \in R_u \times_{R/I_{uv}} R_v$ such that $\pi_{v1}((a,a)) = a$ and $(a,a) \in R_v \times_{R/I_{vw}} R_w$ such that $\pi_{v2}((a,a)) = a$. We know $(a,a)$ is in each first level fiber product since the constant spline satisfies the GKM-condition in each case. Therefore $\pi_{v1}$ and $\pi_{v2}$ are indeed surjective, hence the second level fiber product is a pullback along these surjective morphisms. Thus the theorem's claim is also true for $n=2$ edges. This establishes the base cases.

For the inductive step, suppose that for every connected edge-labeled graph $G$ with $n=k$ edges, the associated ring of splines $R_{G}$ is obtained by finitely many pullbacks or equalizers of surjections. We will show each $R_{G'}$ with connected $G'$ having $n=k+1$ edges is also obtained after finitely many pullbacks or equalizers of surjections. To this end, let $G'$ be any connected edge-labeled graph with $n=k+1$ edges.

Since $G'$ is connected with $n=k+1$ edges, notice $G'$ can be obtained by taking a graph $G$ with $n=k$ edges and extending in one of two ways; the first way is to add a new vertex and then connect it to an existing vertex of $G$, and the second way is to add a new edge to $G$ between existing vertices. We will show that in either case, the fact $G$ satisfies the property guarantees that the extension $G'$ must also satisfy the property. We begin with case 1.

\textbf{Case 1:} Let $G'$ be a graph with $n=k+1$ edges obtained by introducing a new vertex to $G$ and adding an edge between it and an existing vertex of $G$, with the new vertex object called $R_u$ and the existing one called $R_v$. We know that $R_{G}$ has a projection onto $R_v$. Thus, we may take the pullback over the edge diagram $R_u \rightarrow R/I_{uv} \leftarrow R_v$, and notice that there is now a common projection onto $R_v$. Now pulling back along these common projections to $R_v$ we obtain the fiber product $R_{G} \times_{R_v} (R_u \times_{R/I_{uv}} R_v)$. The commutative diagram is shown in Figure~\ref{fig:case1-pullback}.

\begin{figure}
\centering
\[
\begin{tikzcd}[ampersand replacement=\&, column sep=large, row sep=large]
  R_{G} \times_{R_v}\!\bigl(R_u \times_{R/I_{uv}} R_v\bigr)
    \arrow[d, "\pi_G"'] \arrow[r, "\pi_{uv}"]
    \arrow[dr, phantom, very near start]
  \& R_u \times_{R/I_{uv}} R_v
    \arrow[d, "\pi_v"'] \arrow[r, "\pi_u"]
    \arrow[dr, phantom, very near start]
  \& R_u \arrow[d, "\alpha"] \\
  R_{G} \arrow[r, "\pi^G_v"']
  \& R_v \arrow[r, "\beta"']
  \& R/I_{uv}
\end{tikzcd}
\]
\caption{Case 1: pullback along the common projection to \(R_v\), yielding \(R_{G} \times_{R_v}\bigl(R_u \times_{R/I_{uv}} R_v\bigr)\).}
\label{fig:case1-pullback}
\end{figure}

We can see $R_{G} \times_{R_v} (R_u \times_{R/I_{uv}} R_v)$ is the limit since the pairs $(u,v) \in R_u \times_{R/I_{uv}} R_v$ selected by the fiber product at the top left of the diagram are those satisfying the GKM-condition for both the new edge diagram and the other edge diagrams that may connect to $R_v$, while projections factoring through $R_{G}$ satisfy all other GKM-conditions. 

For surjectivity notice that for any $a \in R_v$, we can select the constant spline $s \in R_{G}$ which assigns $a$ to every vertex, hence $s$ projects onto $a$ for $a \in R_v$. Since $a \in R_v$ was selected arbitrarily, we can see the projection from $R_{G}$ to $R_v$ is surjective. Similarly, for each $a \in R_v$ we can select $(a,a) \in R_u \times_{R/I_{uv}} R_v$ which projects to $a \in R_v$. Thus, the projection from $R_{G}$ to $R_v$ and from $ R_u \times_{R/I_{uv}} R_v$ to $R_v$ are both surjective.

\textbf{Case 2.} For case 2 we add a new edge between existing vertices, which introduces a new edge diagram $R_u \rightarrow R/I_{uv} \leftarrow R_v$. Now since $R_{G}$ maps to $R_u$ and $R_v$, we obtain a pair of morphisms $R_{G} \;\rightrightarrows\; R/I_{uv}$. Now note that the ring of splines $R_{G'}$ for the extended graph with an additional edge is precisely the equalizer:
\[
R_{G'} \rightarrow R_{G} \;\rightrightarrows\; R/I_{uv}.
\]
Additionally, we can see that $R_{G'}$ is the equalizer of surjections in particular since $R_{G}$ contains the constant splines.

So we can see that for any edge-labeled connected graph $G$ with $n=1$ and $n=2$ edges, the ring of splines $R_{G}$ is obtained after finitely many pullbacks or equalizers of surjections. We can also see that if for every edge-labeled connected graph $G$ with $n=k$ edges the $R_{G}$ are obtained by finitely many pullbacks or equalizers of surjections, then the same is true for every $R_{G'}$ with $G'$ having $n=k+1$ edges. Therefore by the principle of mathematical induction we conclude all rings of splines for connected $G$ are obtained by finitely many pullbacks or equalizers of surjections.
\end{proof}

\section{The scheme-theoretic view}\label{sec:pushouts}
Now that the algebraic preliminaries are in place, we can interpret generalized splines from the perspective of schemes. In particular, our primary goals are to understand the GKM-condition from the view of schemes, to understand the relationship between $\Spec(R)$ and $\Spec(R_G)$, and finally to see the relationship between these two former goals in terms of sections of $\Spec(R_G)$. The last task of the section is to define the natural $R$-scheme structure possessed by $\Spec(R_G)$. Taken together these four objectives will provide a somewhat complete interpretation of generalized splines in terms of schemes.  The next section will collect foundational results.

The GKM-condition which states adjacent vertex labels agree modulo edge ideals translates to the language of graph morphisms of schemes as follows. Let $R_u \rightarrow R/I_{uv} \leftarrow R_v$ be an edge diagram. A pair $(f, g) \in R_u \times_{R/I_{uv}} R_v$ satisfies $f \equiv g \bmod I_{uv}$, which just says the scheme-theoretic evaluations of $f$ and $g$ over the points of the closed subscheme $\Spec(R/I_{uv}) \hookrightarrow \Spec(R)$ agree. That is, we have the commutative diagram for the graph morphisms of $f$ and $g$:

\[
\begin{tikzcd}[ampersand replacement=\&, column sep=large]
  \Spec(R/I_{uv})
    \arrow[r, hook, "e"]
  \&
  \Spec(R)
    \arrow[r, shift left=0.8ex, "f"]
    \arrow[r, shift right=0.8ex, "g"']
  \&
  \Spec(R[y])
\end{tikzcd}
\]

This says the graph morphisms induced by $f$ and $g$ interpolate over the points of $\Spec(R/I_{uv})$. This now justifies the fact that the GKM-condition was precisely formulated to generalize the notion of a spline, which creates a piecewise function from functions that agree over given points. Thus, the GKM-condition when translated to the language of graph morphisms precisely recovers the original notion of piecing together graphs that agree over various knots. We need to be cautious however, since sections of schemes are not determined by their values. That is, because the structure sheaf may encode infinitesimal data, two sections could agree over all points but still not be equal.

Consider an edge-labeled graph over $\R[x]$ with two vertices and one edge which is labeled with $\langle x-1 \rangle$. If we label the vertices with $f$ and $g$ respectively where they agree modulo $\langle x-1 \rangle$, then the images of the graph morphisms and the commutative diagram can be visualized as in Figure~\ref{fig:graph-morphism-basic}.

\begin{figure}
  \centering
  \includegraphics[width=0.95\linewidth]{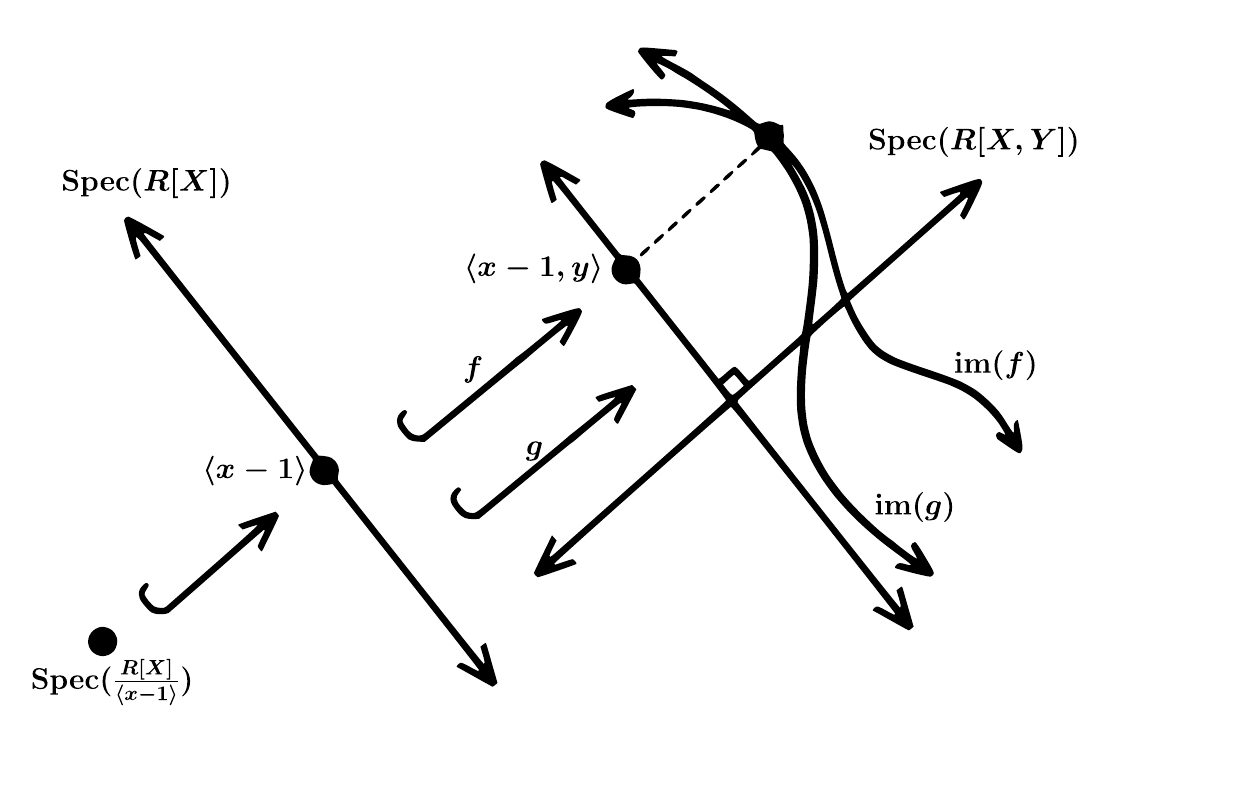}
  \caption{Basic graph morphism diagram.}
  \label{fig:graph-morphism-basic}
\end{figure}

Here the graph morphism $f : \Spec(R) \rightarrow \Spec(R[y])$ becomes $f : \Spec(\R[x]) \rightarrow \Spec(\R[x,y])$ since $R=\R[x]$. We can see that the images of the morphisms $f$ and $g$ agree over $\langle x-1 \rangle$. In the diagram we can also see $\langle x - 1 \rangle$ embedded as a closed subscheme. Indeed, it is geometrically clear that the injection of $\Spec(\R[x] / \langle x - 1 \rangle)$ into $\Spec(\R[x])$ equalizes the two graph morphisms.

If we instead consider a two-vertex edge-labeled graph over $\Z$ with edge label $\langle 3 \rangle$, the same phenomenon shows up as depicted in Figure~\ref{fig:graph-morphisms-arithmetic}. Thus, even in the purely arithmetic case, the GKM-condition can be interpreted in terms of interpolation of graph morphisms over closed subschemes.

\begin{figure}
  \centering
  \includegraphics[width=0.95\linewidth]{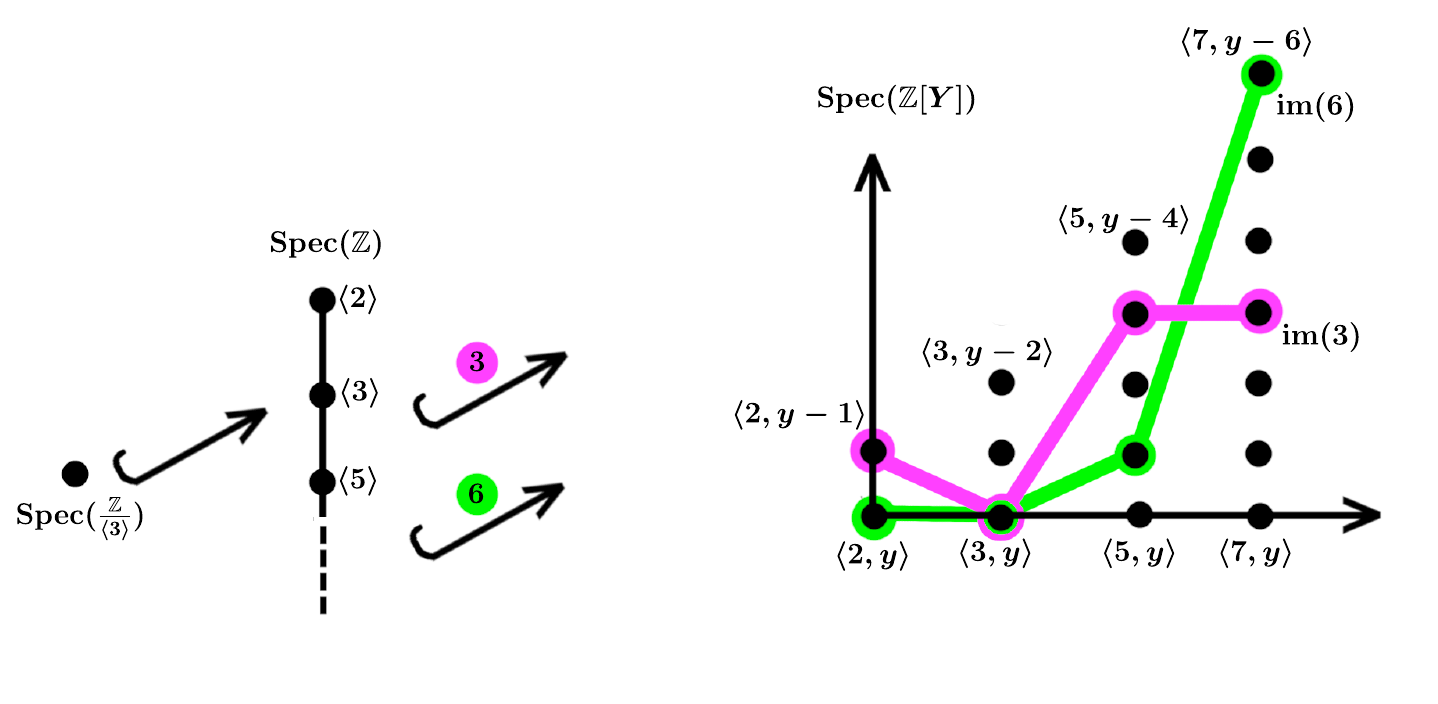}
  \caption{Arithmetic graph morphisms diagram.}
  \label{fig:graph-morphisms-arithmetic}
\end{figure}

A graph spline then is a collection of scheme-theoretic functions which interpolate over the prescribed closed subschemes. It follows that $\Spec(R_{G})$ ought to be some space over which the graph morphisms naturally interpolate according to the GKM-condition. We will now describe $\Spec(R_{G})$ and show that it behaves as expected. Consider the edge diagram:

\[
R_u \rightarrow R/I_{uv} \leftarrow R_v.
\]

Since both morphisms are surjective we have a corresponding pair of closed immersions in the category of schemes:

\[
\Spec(R_u) \leftarrow \Spec(R/I_{uv}) \rightarrow \Spec(R_v)
\]

And we can see the fiber product of the first diagram corresponds to the Ferrand pushout of the second diagram of schemes, which is interpreted geometrically as gluing two copies of $\Spec(R)$ along $\Spec(R/I_{uv})$ since both morphisms correspond to closed immersions. This situation is visualized in Figure~\ref{fig:single-pushout}.

\begin{figure}
  \centering
  \includegraphics[width=\linewidth]{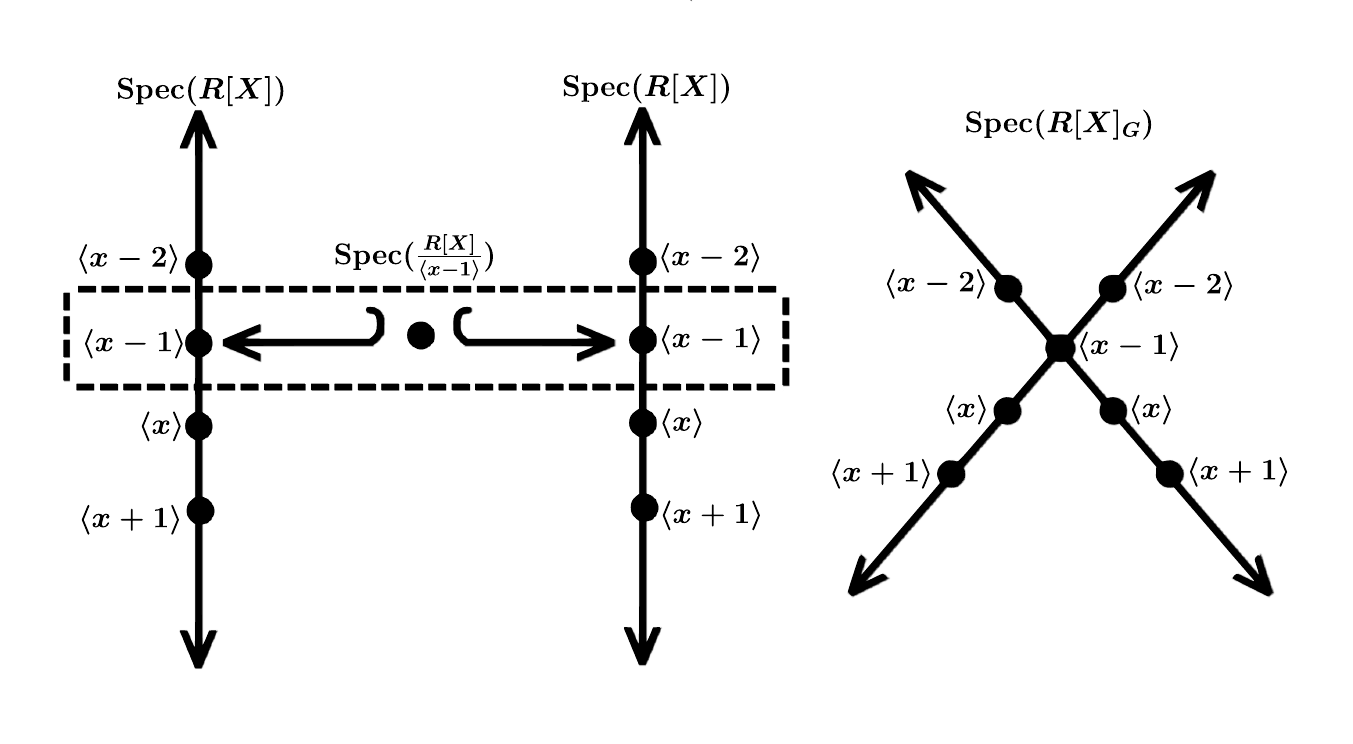}
  \caption{A single pushout diagram. Here $R_u=R_v=R[x]$ and $I_{uv}=\langle x - 1 \rangle$.}
  \label{fig:single-pushout}
\end{figure}

As an additional example consider the edge-labeled graph over $\R[x]$ given in Figure~\ref{fig:rx-cycle}, which consists of a single cycle with edge labels $\langle x \rangle$, $\langle x-1 \rangle$, and $\langle x-2 \rangle$. This corresponds to the gluing situation in Figure~\ref{fig:cycle-pushout}. 
Since graph morphisms are present for sections of any scheme, we have a similar geometric interpretation for the following arithmetic situation. Consider the edge-labeled graph over $\Z$ with edge ideals $\langle 3 \rangle$, $\langle 5 \rangle$, and $\langle 7 \rangle$ shown in Figure~\ref{fig:z-cycle}. This corresponds to the gluing situation in Figure~\ref{fig:integer-cycle}.

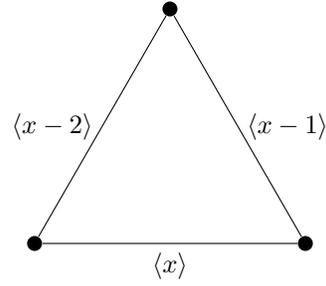
\begin{figure}
\centering
\begin{tikzpicture}[scale=1.2,
  vertex/.style={circle,fill=black,inner sep=2pt}] 
  \node[vertex] (A) at (0,0) {};
  \node[vertex] (B) at (3,0) {};
  \node[vertex] (C) at (1.5,2.6) {};
  \draw (A) -- node[below] {$\langle x\rangle$} (B);
  \draw (B) -- node[right] {$\langle x-1\rangle$} (C);
  \draw (C) -- node[left]  {$\langle x-2\rangle$} (A);
\end{tikzpicture}
\caption{A cycle in $\R[x]$ with edge ideals $\langle x\rangle,\langle x-1\rangle,\langle x-2\rangle$.}
\label{fig:rx-cycle}
\end{figure}

\begin{figure}
  \centering
  \includegraphics[width=0.95\linewidth]{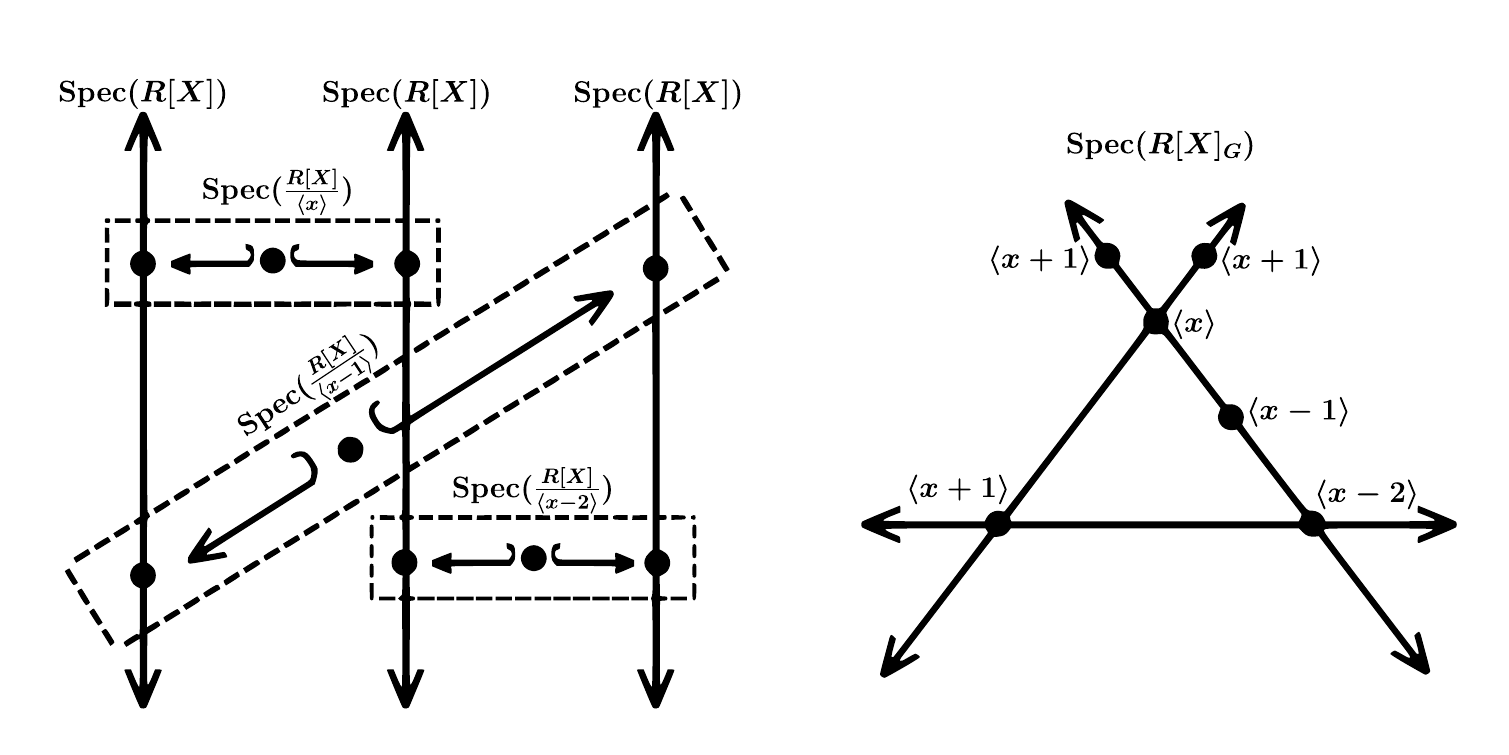}
  \caption{A cycle pushout diagram.}
  \label{fig:cycle-pushout}
\end{figure}

\begin{figure}
\centering
\begin{tikzpicture}[scale=1.2,
  vertex/.style={circle,fill=black,inner sep=2pt}]  
  \node[vertex] (A) at (0,0) {};
  \node[vertex] (B) at (3,0) {};
  \node[vertex] (C) at (1.5,2.6) {};
  \draw (A) -- node[below] {$\langle 3 \rangle$} (B);
  \draw (B) -- node[right] {$\langle 5 \rangle$} (C);
  \draw (C) -- node[left]  {$\langle 7 \rangle$} (A);
\end{tikzpicture}
\caption{An integer-labeled cycle over $\Z$.}
\label{fig:z-cycle}
\end{figure}
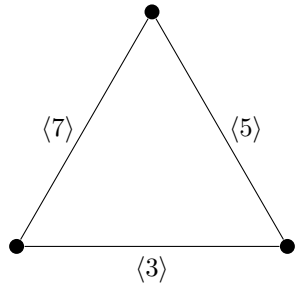

\begin{figure}
  \centering
  \includegraphics[width=0.95\linewidth]{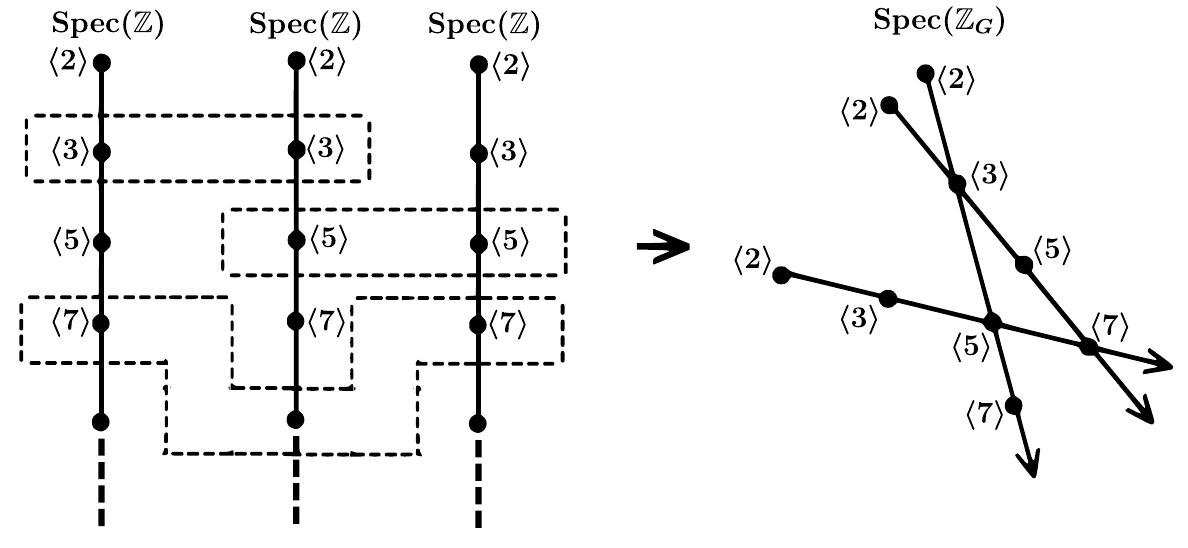}
  \caption{An integer cycle diagram.}
  \label{fig:integer-cycle}
\end{figure}

Of course, we need to justify that this gluing process does not, at any point, fail due to some pathology. It turns out that it does not: this is precisely the scheme-theoretic version of Theorem~\ref{thm:iterated-pullback}. Because the ring of splines is obtained by finitely many pullbacks or equalizers of surjections, the spectrum of the ring of splines is obtained by finitely many gluing operations along closed immersions of schemes. Thus, we can indeed interpret:

\[
\Spec(R_u) \leftarrow \Spec(R/I_{uv}) \rightarrow \Spec(R_v) \leftarrow \Spec(R/I_{vw}) \rightarrow \Spec(R_w)
\]

As being obtained by gluing three copies of $\Spec(R)$ along the closed immersions induced by the edge labeling. The proof of Theorem~\ref{thm:iterated-pullback}, in the base case for $n=2$ edges says that we take four copies of $\Spec(R)$, glue according to the GKM-condition, and then two of the four copies of $\Spec(R)$ are glued into one copy. That is, we once again consider Figure~\ref{fig:second-level-pullback}. On the scheme side, we have the diagram in Figure~\ref{fig:schemes-second-level-pushout}.

\begin{figure}
\centering
\begin{tikzcd}[row sep=4em, column sep=1.2em]
  &[-0.4em] {} &
  \makebox[0pt]{\(\Spec\!\Bigl((R_u \times_{R/I_{uv}} R_v)\!\times_{R_v}\!(R_v \times_{R/I_{vw}} R_w)\Bigr)\)}
  &[-0.4em] {} & \\[1.0em]
  &
  \Spec\!\bigl(R_u \times_{R/I_{uv}} R_v\bigr)
  \arrow[ur,"\Spec(\pi_1)", shorten <=2pt, shorten >=2pt]
  &[-0.6em] {} &[-0.6em]
  \Spec\!\bigl(R_v \times_{R/I_{vw}} R_w\bigr)
  \arrow[ul,"\Spec(\pi_2)"', shorten <=2pt, shorten >=2pt]
  & \\[0.3em]
  \Spec(R_u)
  \arrow[ur,"\Spec(\pi_{u})", shorten <=2pt, shorten >=2pt]
  &
  \Spec(R/I_{uv})
  \arrow[l] \arrow[r]
  &[-0.6em]
  \Spec(R_v)
  \arrow[ul,"\Spec(\pi_{v1})"', shorten <=2pt, shorten >=2pt]
  \arrow[ur,"\Spec(\pi_{v2})", shorten <=2pt, shorten >=2pt]
  \arrow[l] \arrow[r]
  &[-0.6em]
  \Spec(R/I_{vw})
  \arrow[l] \arrow[r]
  &
  \Spec(R_w)
  \arrow[ul,"\Spec(\pi_{w})", shorten <=2pt, shorten >=2pt]
\end{tikzcd}
\caption{Dual picture (schemes): the second-level pushout corresponding to the ring pullback in~\ref{fig:second-level-pullback}.}
\label{fig:schemes-second-level-pushout}
\end{figure}
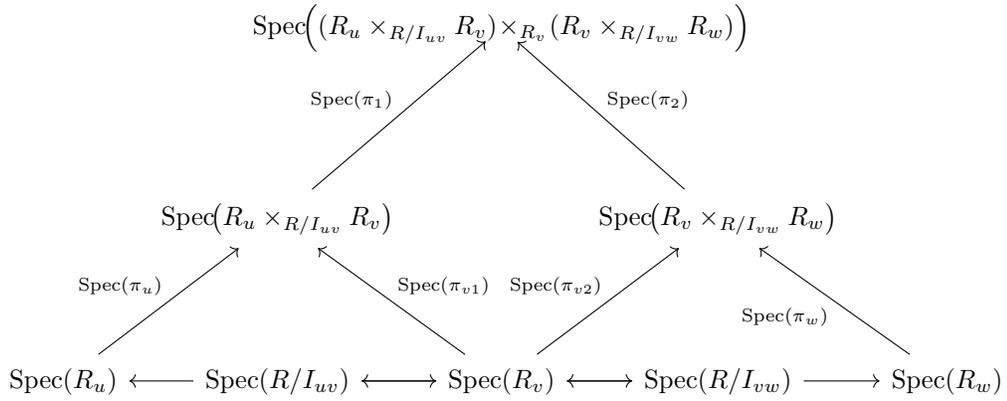

Notice that the first pushouts result in two separate copies of $\Spec(R_v)$, whereas the final pushout is along $\Spec(R_v)$ which effectively glues the two copies together into one. Concretely, suppose we have the edge-labeled graph over $\R[x]$ in Figure~\ref{fig:rx-path}, which has ring of splines $R[x]_{G}$. This translates to the spline diagram of schemes:

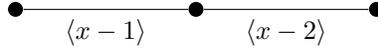
\begin{figure}
\centering
\begin{tikzpicture}[scale=1.2,
  vertex/.style={circle,fill=black,inner sep=2pt}]
  \node[vertex] (A) at (0,0) {};
  \node[vertex] (B) at (2,0) {};
  \node[vertex] (C) at (4,0) {};
  \draw (A) -- node[below] {$\langle x-1\rangle$} (B);
  \draw (B) -- node[below] {$\langle x-2\rangle$} (C);
\end{tikzpicture}
\caption{A path of two edges over $\R[x]$ with ideals $\langle x-1\rangle$ and $\langle x-2\rangle$.}
\label{fig:rx-path}
\end{figure}

\[
\Spec(\R[x]) \leftarrow \Spec(\R[x]/\langle x-1 \rangle) \rightarrow \Spec(\R[x]) \leftarrow \Spec(\R[x]/\langle x-2\rangle) \rightarrow \Spec(\R[x])
\]

where the colimit of the diagram, obtained by successive gluing of schemes, corresponds to $\Spec(\R[x]_{G})$. The two-step gluing process represented by base case 2 of Theorem~\ref{thm:iterated-pullback} will start with four copies of $\Spec(\R[x])$, glue into two pairs, and then glue along the duplicated $\Spec(\R[x])$. The gluing process is depicted in Figure~\ref{fig:base-case-2}, where the subscripts $u$, $v$, and $w$ indicate which object of the spline diagram is in play.

\begin{figure}
  \centering
  \includegraphics[width=1\linewidth]{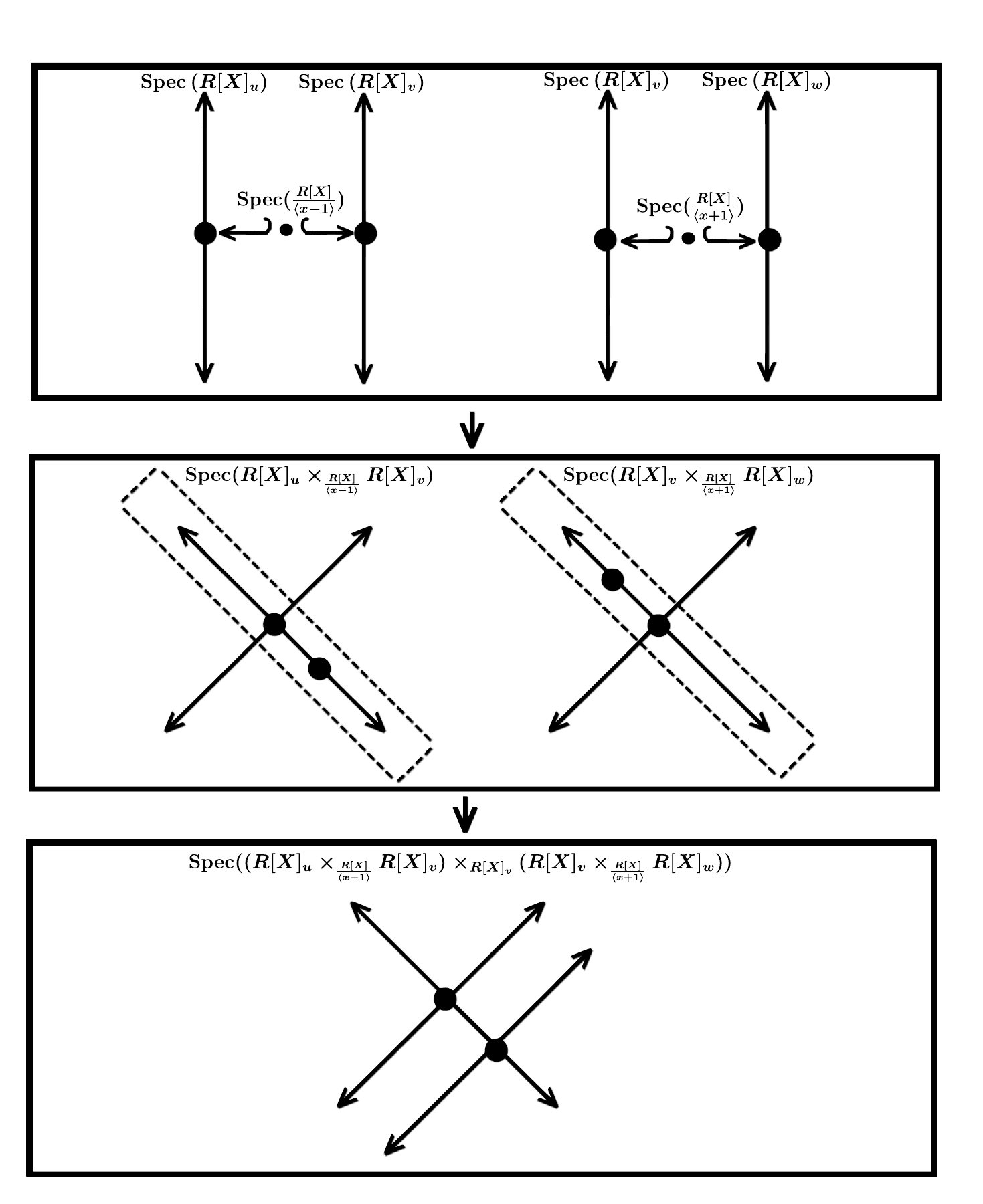}
  \caption{Base case 2 diagram.}
  \label{fig:base-case-2}
\end{figure}

In inductive case 1 which appears in the proof of Theorem~\ref{thm:iterated-pullback}, we can interpret the commutative diagram in Figure~\ref{fig:case1-pullback} as a two-step gluing process. This is visualized in Figure~\ref{fig:inductive-case1} where we have a cycle similar to the one in Figure~\ref{fig:cycle-pushout} which has been extended to include a new vertex.
 
\begin{figure}
  \centering
  \includegraphics[width=0.95\linewidth]{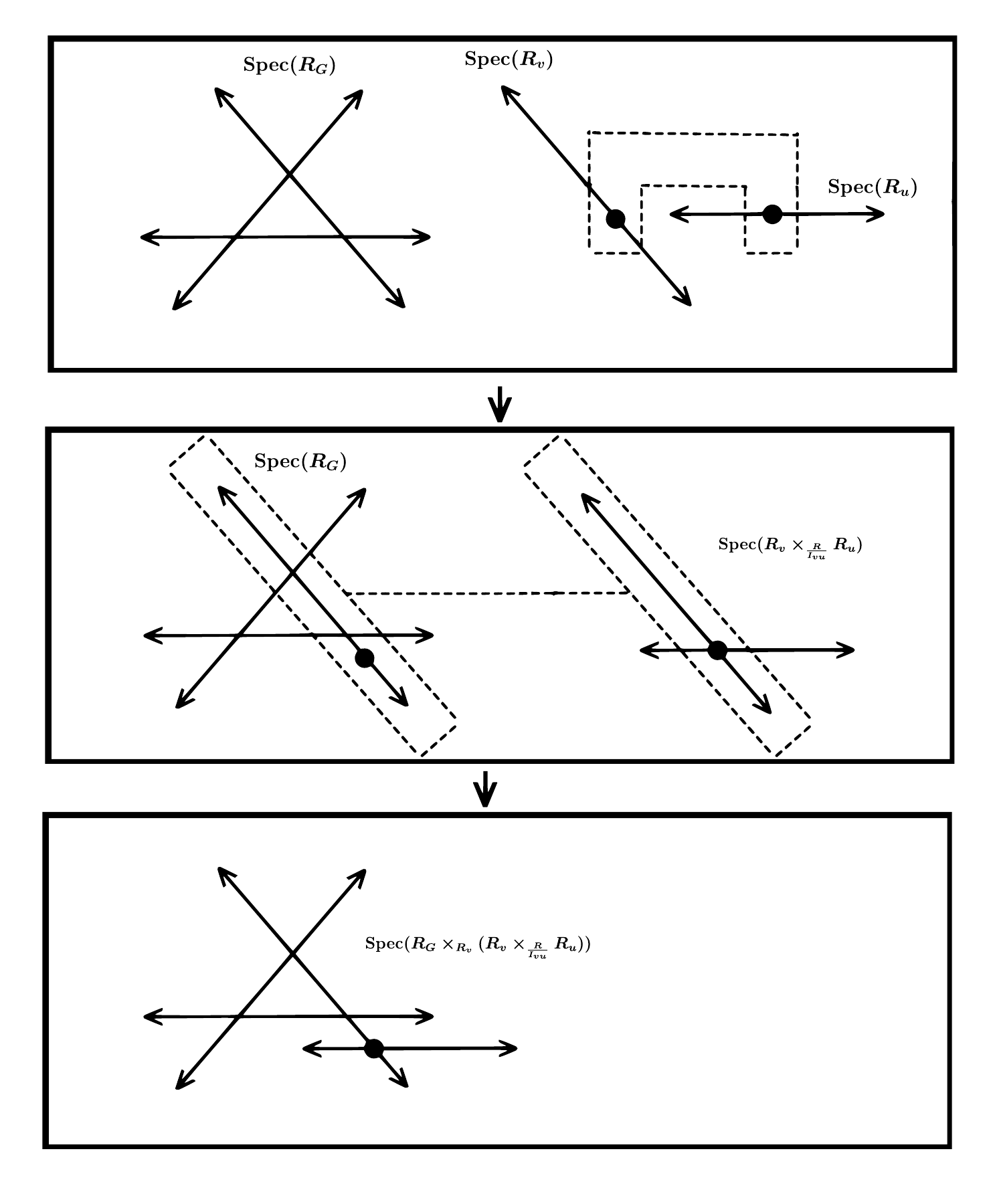}
  \caption{Inductive case 1 diagram.}
  \label{fig:inductive-case1}
\end{figure}

Thus, adding a new vertex and edge to an existing vertex is tantamount to gluing a new copy of $\Spec(R)$ (corresponding to the new vertex) to $\Spec(R_{G})$ at the existing $\Spec(R)$ determined by the existing vertex, at the point determined by the new edge. 

In inductive case 2 where an edge was added between existing vertices, this should correspond to existing copies of $\Spec(R)$ contained in $\Spec(R_{G})$ being glued together along a new closed subscheme $\Spec(R/I)$. Indeed, algebraically we took the equalizer \(
R_{G}' \rightarrow R_{G} \;\rightrightarrows\; R/I_{uv}\), which on the scheme side corresponds to the coequalizer 

\[
\Spec(R_{G}') \leftarrow \Spec(R_{G}) \;\leftleftarrows\; \Spec(R/I_{uv}),
\]

Here a single $\Spec(R_{G})$ is fixed, and the coequalizer must ensure both images of $\Spec(R/I_{uv})$ inside $\Spec(R_{G})$ are identified. And in particular, we know $\Spec(R/I_{uv})$ will inject into the copies of $\Spec(R_u)$ and $\Spec(R_v)$ contained in $\Spec(R_{G})$. Thus, in case 2 two existing copies of $\Spec(R)$ inside of $\Spec(R_{G})$ are glued. 

This aligns with the view that each vertex of an edge-labeled graph corresponds to a copy of $\Spec(R)$, and that the edges describe how those copies of $\Spec(R)$ glue. In summary, both case 1 and case 2 in the proof of Theorem~\ref{thm:iterated-pullback} have a natural geometric explanation.

So far, we have translated the algebraic nature of the base and inductive cases of Theorem~\ref{thm:iterated-pullback} into geometric language, and shown it matches the intuitive idea of gluing one copy of $\Spec(R)$ for each vertex of an edge-labeled graph along the closed subschemes determined by the edges. The algebra was needed to prove the intuitive perspective is pathology free insofar as working with affine schemes is concerned.

The sections of $\Spec(R_{G})$ are now sections along the glued up copies of $\Spec(R)$, where each vertex label $s(u)$ lives over the copy of $\Spec(R_u)$ it is associated with; for adjacent $s(v)$ over $\Spec(R_v)$, the graphs of $s(u)$ and $s(v)$ agree over the points of $\Spec(R/I_{uv})$ which both glued copies of $\Spec(R)$ now share. Therefore we can visualize sections of $\Spec(R_{G})$ by adding a $y$-axis along the geometric picture of $\Spec(R_{G})$ and plotting the graphs over each $\Spec(R)$, taking care to ensure they agree over the places the $\Spec(R)$ are glued up. If $\Spec(R_{G})$ is two glued up copies of the affine real line we can visualize a graph morphism induced by an element of $R_{G}$ as shown in Figure~\ref{fig:RG-graph-morphism}.

\begin{figure}
  \centering
  \includegraphics[width=1\linewidth]{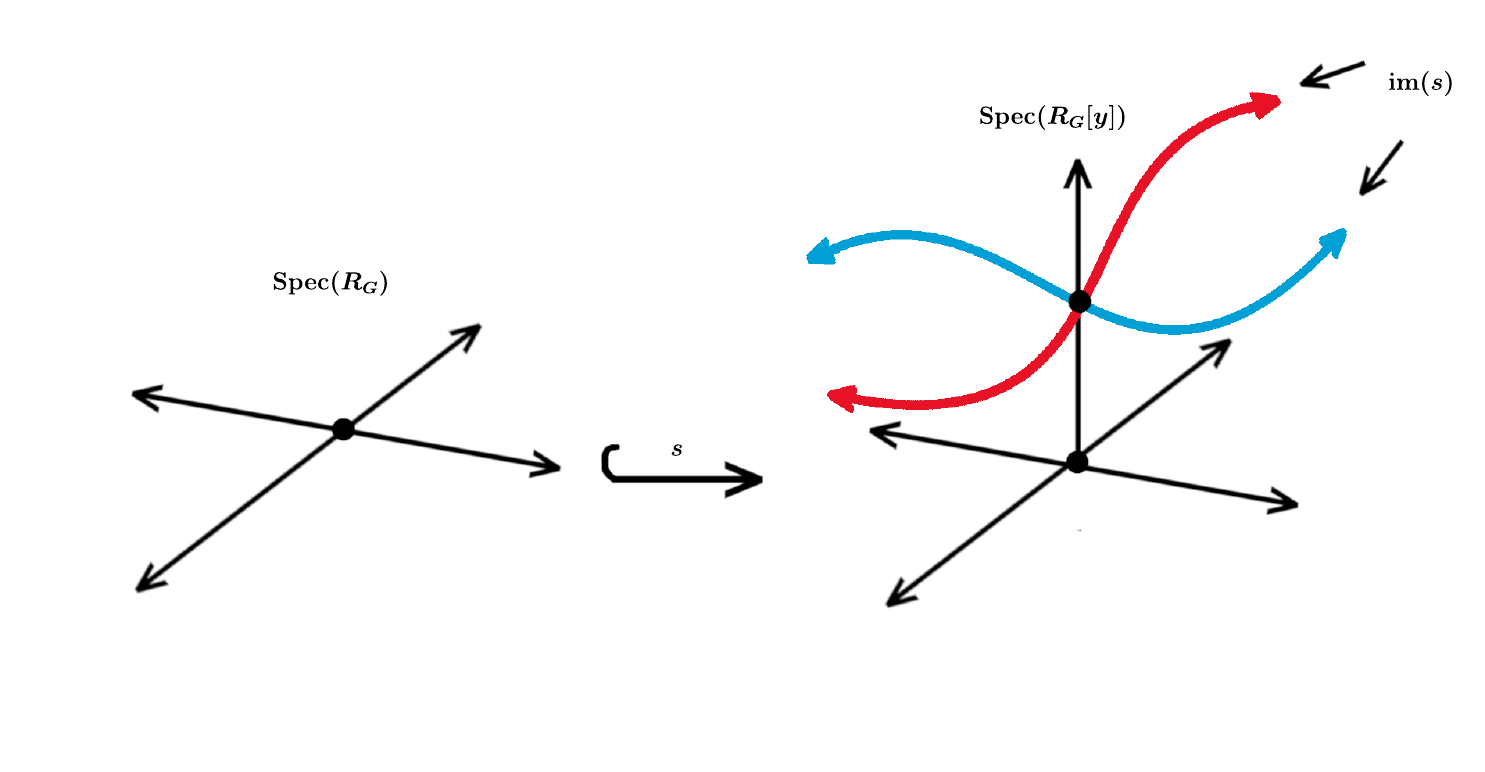}
  \caption{Graph morphism diagram for \(R_{G}\).}
  \label{fig:RG-graph-morphism}
\end{figure}

That is, if we assume $R=\R[x]$ and that $G$ is a graph with one edge-labeled by some $\langle x - a \rangle$, for the codomain of the graph morphism we have two $xy$-planes intersecting along a vertical line; the graph of a spline $s$ corresponds to a curve over the first $xy$-plane, and a curve over the second, such that they agree where the two planes meet. The space the graph lives in should not be confused with an $xyz$-coordinate space, however, because points only live on the two planes and their shared line of intersection, but do not exist off the planes.

In summary, we first interpreted the GKM-condition as equalizers of graph morphisms. The algebraic perspective given by ideals is complemented by this interpretation, since it shows the intuitive notion of piecewise functions agreeing along prescribed points is directly generalized geometrically, not just algebraically. Second, we showed that $\Spec(R_G)$ is obtained by gluing copies of $\Spec(R)$ along closed subschemes determined by the edge labels. Finally, we used these first two observations to understand the graph morphisms attached to sections of $\Spec(R_G)$. We conclude by giving $R_G$ the structure of an $R$-scheme.

\begin{definition}  
Let $R$ be any commutative ring, $G$ an edge-labeled graph, and $R_{G}$ the associated ring of splines. Let $\Spec(R_{G})$ be given an $R$-scheme structure induced by the diagonal embedding $\Delta : R \rightarrow R_G$ mapping each element of $R$ to the constant spline. We write $G_\Delta$ to denote the \textbf{spline structure morphism} which gives $\Spec(R_G)$ the structure of an $R$-scheme.
\end{definition}

\section{Foundational results}
In this section we will collect foundational results needed to work with spline structure morphisms. First we will consider a functor $F_S(-) = - \otimes_R S$ with $S$ flat, and define an edge-labeled graph $F_S(G)$ by first applying $F_S$ to the spline diagram of rings $\mathfrak G$, and then building $F_S(G)$ from that. Next, we will establish the stability of $\Spec(R_G)$ as a colimit of a spline diagram of schemes under flat base change. The flat base change result  essentially asserts that both of the following
\[ 
F_S(G_\Delta): \Spec(F_S(R_G)) \rightarrow  \Spec(F_S(R)) 
\]
\[
F_S(G)_\Delta: \Spec(F_S(R)_{F_S(G)}) \rightarrow \Spec(F_S(R))
\]
are equivalent. Here $F_S(R)_{F_S(G)}$ enjoys the full privileges of a ring of generalized splines over $F_S(R)$ with edge-labeled graph $F_S(G)$, while $F_S(R_G)$ enjoys status as a flat base change of $R_G$. Thus, while we could always study a ring of splines after base change, the base change is now upgraded to a full ring of splines. Effectively, flat base change induces a transformation of spline structure morphisms. The last of the three basic techniques we will examine is that of lifting graphs. Essentially, given an edge-labeled graph $G$, we want to find $G'$ such that $F_S(G')=G$. Taken together, these results will power the remaining proofs in the paper.

\begin{lemma}
\label{lma:graph-pushforward}
Let $R$ be a commutative ring, $G$ an edge-labeled graph, and $\mathfrak G$ the associated spline diagram of rings. Suppose $F_S(-) = - \otimes_R S$ with $S$ flat. Then there exists an edge-labeled graph denoted $F_S(G)$ over $F_S(R)$ with associated spline diagram of rings $F_S(\mathfrak G)$.
\end{lemma}
\begin{proof}
Let $R$, $G$, $\mathfrak G$, and $F_S$ be given as specified. Since $F_S$ is a functor we know $F_S(\mathfrak G)$ acts on edge diagrams as follows:

\[
R_u \rightarrow R_u/I_{uv} \leftarrow R_v
\]
\[
F_S(R_u) \rightarrow F_S(R/I_{uv}) \leftarrow F_S(R_v).
\]
Now in particular, since $F_S$ is given by tensor product with $S$ flat we have the following exact sequences of $R$-modules

\[
0 \rightarrow I_{uv} \rightarrow R_u \rightarrow R/I_{uv} \rightarrow 0
\]
\[
0 \rightarrow F_S(I_{uv}) \rightarrow F_S(R_u) \rightarrow F_S(R/I_{uv}) \rightarrow 0.
\]

Hence each edge diagram under $F_S$ is of the form

\[
F_S(R_u) \rightarrow F_S(R)/F_S(I_{uv}) \leftarrow F_S(R_v) 
\]

Now let $G$ be the original edge-labeled graph, and take $R$ to $F_S(R)$ and $F_S(G)$ to the edge-labeled graph obtained by replacing each edge ideal $I_{uv}$ with $F(I_{uv})$. Then by the previous observation, we can see $F_S(G)$ has $F_S(\mathfrak G)$ as its associated ring of limit splines. 
\end{proof}

\begin{definition}
Let $F$ be a functor, $R$ a commutative ring, $G$ an edge-labeled graph, and $R_G$ the ring of splines. If the following are equivalent:

\[ 
F(G_\Delta) : \Spec(F(R_G)) \rightarrow \Spec(F(R))
\]

\[
F(G)_\Delta : \Spec(F(R)_{F(G)}) \rightarrow \Spec(F(R))
\]

then we say $F$ preserves (or commutes with) spline structure morphisms.
\end{definition}

\begin{definition}
Let $F$ be a functor of commutative rings which also preserves module freeness. Suppose $G'$ is an edge-labeled graph over $F(R)$ with ring of generalized splines $F(R)_{G'}$. Then any edge-labeled graph $G$ over $R$ is said to be a lift of $G'$ if $F(G)=G'$. If $R_G$ is free, then $G'$ is said to be \textbf{free-liftable}. 
\end{definition}

\begin{lemma}
\label{lma:free-lift}
Free-liftable implies free.
\end{lemma}
\begin{proof}
The definition of free-liftable along a functor has the preservation of freeness as a hypothesis. The claim follows.
\end{proof}

For the remainder of the paper, we will focus on the functors of the form $F_S(-) = - \otimes_R S$. The following lemma will be used to prove stability under flat base change to a standard open. Note that we write $R_X$, $R_Y$, and $R_{XY}$ instead of the usual $R_u$, $R_v$, and $R_{uv}$. This is because the rings involved may come from equalizer or pullback steps, and therefore not correspond to a vertex or edge object within the spline diagram of rings.

\begin{lemma}
\label{lma:gluing-lemma}
$\Spec(R_{G})$ is obtained by finitely many Ferrand pushouts or coequalizers of closed immersions.
\end{lemma}
\begin{proof}

We already showed that the colimit $\Spec(R_{G})$ is obtained by finitely many pushouts or coequalizers of closed immersions, since these are the dual statements of Theorem~\ref{thm:iterated-pullback}. It remains to show the pushouts are Ferrand. In particular, for any pushout along $\Spec(R_{XY})$ we need to show $i: \Spec(R_{XY}) \rightarrow \Spec(R_X)$ is a closed immersion, that $j: \Spec(R_{XY}) \rightarrow \Spec(R_Y)$ is finite integral, and that for $y \in \Spec(R_Y)$, $j^{-1}(\{y\}) \subseteq i^{-1}(U)$ for some open $U \subseteq \Spec(R_X)$ \cite{StacksPushoutsII}.

We know that both morphisms for the pushout are closed immersions, which are known to be finite and integral per Stacks Project Lemma 29.44.12 \cite{StacksIntegralFinite}. Therefore it remains to verify the pullback property. Since $\Spec(R_{XY})$ sits inside both $\Spec(R_X)$ and $\Spec(R_Y)$ we can select any open $U$ in $\Spec(R_X)$ containing $i(j^{-1}(y))$ and we obtain that $j^{-1}({y}) \subseteq i^{-1}(U)$ as desired.
\end{proof}

We need the next definition to keep notation in the main theorem and subsequent applications under control.

 \begin{definition}
Let $R$ be a commutative ring with edge-labeled graph $G$. For any standard open set $U \subseteq \Spec(R)$, we write $R_U := \mathcal{O}_{\Spec(R)}(U)$ for the localized base ring on $U$, and we write $G_U$ for the restriction of $G$ to $U$ (obtained by applying the base change functor $-\otimes_R R_U$ to the spline diagram of $G$). If $U$ is a principal open $D(f)$, we also write $R_f$ and $G_f$. The ring of splines of $G_U$ over $R_U$ will be written $(R_U)_{G_U}$. For a principal open $D(f)$ this is $(R_f)_{G_f}$.
\end{definition}

We now obtain the following consequence of the lemma and Theorem~\ref{thm:iterated-pullback}:

\begin{theorem}[Main Theorem]
\label{thm:flat-base-change} 
Flat base change with affine schemes commutes with spline structure morphisms. In particular, suppose $R$ is a commutative ring, $G$ an edge-labeled graph, and $R_G$ the ring of splines. If $\mathfrak p$ is a prime ideal of $R$, then $(R_G)_{\mathfrak p} \cong (R_{\mathfrak p})_{G_{\mathfrak p}}$, and if $f \in R$, then $(R_G)_f \cong (R_f)_{G_f}$. 
\end{theorem}
\begin{proof} 
Ferrand pushout is known to be stable under flat base change. Since $\Spec(R_{G})$ is both the colimit of the spline diagram $\mathfrak G$ and is also obtained by successive Ferrand pushouts or coequalizers of closed immersions (by the lemma), we can see the flat base change of $\Spec(R_{G})$ is the colimit of the flat base change of $\mathfrak G$. This is because each of the Ferrand pushouts and coequalizers used to obtain $\Spec(R_{G})$ are stable under flat base change, so the stability propagates along each of the finite steps to $\Spec(R_{G})$. Provided the base remains affine, then the spline diagram is still of affine schemes and there is an associated edge-labeled graph \footnote{The inductive argument used in Theorem~\ref{thm:iterated-pullback} which underlies this result mirrors Euler's proof that the Euler characteristic of a graph is exactly 2; incidentally, a corollary of Euler's proof is that there are only five platonic solids.}. But that just says flat base change with an affine scheme commutes with spline structure morphisms as desired.
\end{proof}

We now obtain two corollaries that power the local--global results in the applications section. That is, the results in the applications section work by ensuring the criteria of the two corollaries are met.

\begin{corollary}
\label{cor:locally-free}
Let $R$ be a commutative ring, $G$ an edge-labeled graph, and $R_G$ the ring of splines. Let $\{U_f\}$ be a standard open cover of $\Spec(R)$. If each $(R_f)_{G_f}$ is free, then $G$ is said to be \textbf{locally free}. Similarly, if $(R_{\mathfrak p})_{G_{\mathfrak p}}$ is free for each prime ideal $\mathfrak p$ of $R$, then $G$ is also said to be locally free. $R_G$ is locally free if and only if $G$ is locally free. If $R$ is a PID or of the form $A[x_1,\dots,x_n]$ with $A$ a PID or field, then $R_G$ is also free.
\end{corollary}
\begin{proof}
The main theorem implies that $(R_G)_f \cong (R_f)_{G_f}$ as $R_f$-modules, and similarly if we localize at a prime ideal $\mathfrak p$, we have $(R_G)_{\mathfrak p} \cong (R_{\mathfrak p})_{G_{\mathfrak p}}$ as $R_{\mathfrak p}$-modules. Since the latter of each isomorphism is free by hypothesis, so is the former, and therefore $R_G$ is a locally free $R$-module as claimed.
\end{proof}

The corollary captures a critical but easy to overlook consequence of the main theorem. Commutation with spline structure morphisms means that the combinatorial data used to build a ring of splines survives flat base change and continues to fully control the ring of splines. Thus, in order to study a ring of splines, \textbf{it suffices to study the localized combinatorial data, which is often much simpler than the global combinatorial data}.

\begin{corollary}
\label{cor:globally-free-implies-locally-free}
Let $R$ be a commutative ring, $G$ an edge-labeled graph, $R_G$ the ring of splines, and $U$ a standard open. 

\begin{itemize}
\item If $R_G$ is a free $R$-module, then $(R_U)_{G_U}$ is a free $R_U$-module; Conversely, if $(R_U)_{G_U}$ is not a free $R_U$-module, then $R_G$ is not free. 

\item If $R_G$ is a free $R$-module and $\mathfrak p$ is any prime ideal of $R$, then $(R_{\mathfrak p})_{G_{\mathfrak p}}$ is a free $R_{\mathfrak p}$-module; Conversely, if $(R_{\mathfrak p})_{G_{\mathfrak p}}$ is not a free $R_{\mathfrak p}$-module, then $R_G$ is not free.
\end{itemize}
\end{corollary}

\begin{proof}
Let $R$ be a commutative ring, $G$ an edge-labeled graph, and $R_G$ a free $R$-module. The converses follow automatically from the first part of each statement, so we only need to prove the first part of each claim. Since flat base change preserves freeness, $(R_G)_U$ is free, and since flat base change preserves spline structure morphisms, we have $(R_G)_U \cong (R_U)_{G_U}$, which implies $(R_U)_{G_U}$ is free as claimed. Notice the second claim is also about a flat base change, so the second claim follows for the same reason as the first. 
\end{proof}

The first part of Corollary~\ref{cor:globally-free-implies-locally-free} means that if we lose access to a combinatorial result on a standard open set, we can lift the localized (and often simpler) combinatorial data to the original ring in order to establish freeness locally. For example, suppose we have a polynomial ring $R$ with edge-labeled graph $G$; Inversion of even a single element $f$ of $R$ means the local edge-labeled graph $G_f$ is no longer an edge-labeled graph for a ring of polynomials. However, we may be able to lift suitable subgraphs of $G_f$ back to $R$ itself where results about polynomial rings do apply. Together, \textbf{Corollary~\ref{cor:locally-free} means we can obtain global results by patching together simpler localized combinatorial data}, while \textbf{Corollary~\ref{cor:globally-free-implies-locally-free} allows the direct transfer of combinatorial data from the global picture of an edge-labeled graph to the local pictures}.

The converses in Corollary~\ref{cor:globally-free-implies-locally-free} should be used with caution. They can only be used ``downstairs'' where the original ring has been tampered with, and suffers from the fact non-free modules can suddenly become free when operated on. Despite these limitations, the consequence is that \textbf{if any local combinatorial data precludes freeness even once, freeness is precluded globally}.

Before we give an example to the theorem, we will introduce a definition that describes behavior that shows up naturally when working locally, but is largely irrelevant in the global context. Note that historically, edge-labeled graphs do not include edges labeled with $R$, because this imposes no non-trivial GKM-constraints and so can be ignored. It turns out edges with $R$ do show up in the local case naturally, and effectively correspond to certain edges vanishing locally.

\begin{definition}
Let $R$ be a commutative ring with edge labeled-graph $G$, and consider a subgraph $H \subseteq G$. Let $\partial H$ denote the boundary edges of $H$, that is, the edges which have exactly one vertex in $H$ and one vertex not in $H$. If $\partial H$ is empty or all edges in it are trivial (labeled by $R$) then we say $H$ \textbf{partially determines} $G$. The collection of all subgraphs that partially determine $G$ are said to \textbf{fully determine} $G$. We will write $I_e \in \partial H$ to denote the ideal $I$ which labels the edge $e \in \partial H$.
\end{definition}

\begin{theorem}
\label{thm:summands}
Let $R$ be a commutative ring with edge-labeled graph $G$, and let $\{H_i\}$ be the collection of edge-labeled graphs that fully determine $G$. Then 
\[
R_G = \left(\bigoplus_j {R}\right) \oplus \left(\bigoplus_i R_{H_i}\right)
\]
\end{theorem}

where $j$ ranges over the vertices with only trivial edges, and $i$ ranges over the $\{H_i\}$. We call the first and second part \textbf{the trivial summand} and \textbf{the determining summand}. The components of the former and latter are each called \textbf{a trivial summand} and \textbf{a determining summand}.
\begin{proof}
First note that if a vertex has only trivial edges, then there are no GKM-constraints imposed. Thus, the vertex can be labeled by any element of $R$ without restriction. Consequently, we can see $R_G$ contains one direct summand for each such vertex as claimed. Now the vertices of each $H_i$ only have the GKM-constraints imposed precisely by the edges of each $H_i$, and hence the labelings for $H_i$ are precisely those of $R_{H_i}$, and these labelings are independent of the other determining subgraphs. Hence we can see each $H_i$ contributes a summand $R_{H_i}$ to $R_G$. Since this accounts for every vertex, we can see $R_G = (\bigoplus_j {R}) \oplus (\bigoplus_i R_{H_i})$ as claimed. 
\end{proof}

\begin{porism}
\label{por:summand}
Let $R$ be a commutative ring, $G$ and edge-labeled graph. If $H$ partially determines $G$, then $R_H$ is a summand of $R_G$.
\end{porism}
\begin{proof}
Immediate from Theorem~\ref{thm:summands}.
\end{proof}

As an example consider the graph in Figure~\ref{fig:three-hexagons} which is formed by three linking hexagons. Notice that the leftmost hexagon has all edge ideals whose generators are powers of $3$, the second hexagon ideals have powers of $5$ as a factor of their generator, and finally the last has powers of $7$. Now a flat base change to the localization of $\Z$ at $\langle 3 \rangle$, denoted $\Z\!\left[\tfrac13\right]$, gives a Zariski open subset $\Spec(\Z[\frac{1}{3}]) \subseteq \Spec(\Z)$ with $3$ now a unit. Thus, in this case we can perform a flat base change to $\Spec(\Z[\frac{1}{3}])$ by pullback along the open immersion $\Spec(\Z[\frac{1}{3}]) \hookrightarrow \Spec(\Z)$.

\begin{figure}
\centering
\begin{tikzpicture}[scale=2.4,
                    every node/.style={font=\scriptsize},
                    ideal/.style={inner sep=1pt}]
  \coordinate (A1) at (-1,0);
  \coordinate (B1) at (-0.5,  0.866);
  \coordinate (C1) at ( 0.5,  0.866);
  \coordinate (D1) at ( 1,    0);
  \coordinate (E1) at ( 0.5, -0.866);
  \coordinate (F1) at (-0.5, -0.866);
  \coordinate (A2) at ( 0,    0);
  \coordinate (B2) at ( 0.5,  0.866);
  \coordinate (C2) at ( 1.5,  0.866);
  \coordinate (D2) at ( 2,    0);
  \coordinate (E2) at ( 1.5, -0.866);
  \coordinate (F2) at ( 0.5, -0.866);
  \coordinate (A3) at ( 1,    0);
  \coordinate (B3) at ( 1.5,  0.866);
  \coordinate (C3) at ( 2.5,  0.866);
  \coordinate (D3) at ( 3,    0);
  \coordinate (E3) at ( 2.5, -0.866);
  \coordinate (F3) at ( 1.5, -0.866);
  \newcommand{\edge}[4][]{\draw (#2) -- (#3) node[midway,ideal,#1] {#4};}
  \edge[above left]  {A1}{B1}{$\langle 3^1\rangle$}
  \edge[above]       {B1}{C1}{$\langle 3^2\rangle$}
  \edge[above right] {C1}{D1}{$\langle 3^3\rangle$}
  \edge[right]       {D1}{E1}{$\langle 3^4\rangle$}
  \edge[below]       {E1}{F1}{$\langle 3^5\rangle$}
  \edge[below left]  {F1}{A1}{$\langle 3^6\rangle$}
  \edge[above]       {A2}{B2}{$\langle 5^1\rangle$}
  \edge[above]       {B2}{C2}{$\langle 5^2\rangle$}
  \edge[above right] {C2}{D2}{$\langle 5^3\rangle$}
  \edge[right]       {D2}{E2}{$\langle 5^4\rangle$}
  \edge[below]       {E2}{F2}{$\langle 5^5\rangle$}
  \edge[below]       {F2}{A2}{$\langle 5^6\rangle$}
  \edge[above]       {A3}{B3}{$\langle 7^1\rangle$}
  \edge[above]       {B3}{C3}{$\langle 7^2\rangle$}
  \edge[above right] {C3}{D3}{$\langle 7^3\rangle$}
  \edge[right]       {D3}{E3}{$\langle 7^4\rangle$}
  \edge[below]       {E3}{F3}{$\langle 7^5\rangle$}
  \edge[below]       {F3}{A3}{$\langle 7^6\rangle$}
  \foreach \v in {A1,B1,C1,D1,E1,F1,
                  A2,C2,D2,E2,
                  A3,C3,D3,E3,
                  B2,B3,F2,F3}
      \fill (\v) circle (1.4pt);
\end{tikzpicture}
\caption{Three linking hexagons with edge ideal generators divisible by $3$, $5$, and $7$, respectively.}
\label{fig:three-hexagons}
\end{figure}
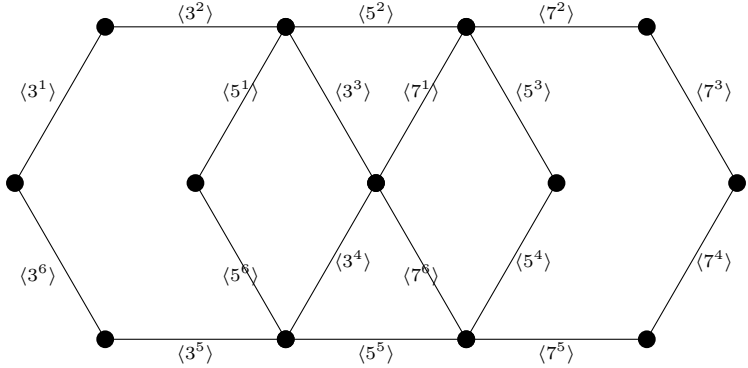

Notice that the tensor product of $\Z/\langle 3 \rangle$ and $\Z[\frac{1}{3}]$ is the zero ring $0$. Hence the associated edge ideals of the form $\langle 3^a \rangle$ must equal $\Z[\frac{1}{3}]$ after flat base change, since $3$ becomes a unit which forces the ideal to be all of $\Z[\frac{1}{3}]$. It follows that a flat base change to $U = \Spec(\Z[\frac{1}{3}])$ will force the GKM-condition for the leftmost hexagon to become trivial. Now if we perform a base change to $\Spec(\Z[\frac{1}{3}][\frac{1}{5}])$, the GKM-condition for all but the rightmost hexagon becomes trivial, as shown in Figure~\ref{fig:basechange-hexagon}.

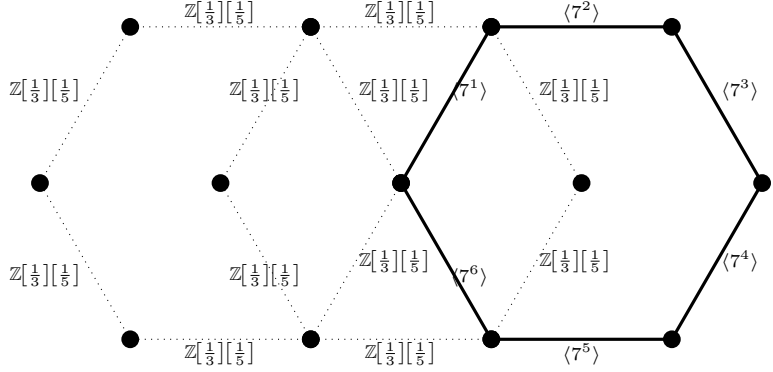
\begin{figure}
\centering
\begin{tikzpicture}[scale=2.4,
                    every node/.style={font=\scriptsize},
                    ideal/.style={inner sep=1pt}]
  \coordinate (A1) at (-1,0);
  \coordinate (B1) at (-0.5,  0.866);
  \coordinate (C1) at ( 0.5,  0.866);
  \coordinate (D1) at ( 1,    0);
  \coordinate (E1) at ( 0.5, -0.866);
  \coordinate (F1) at (-0.5, -0.866);
  \coordinate (A2) at ( 0,    0);
  \coordinate (B2) at ( 0.5,  0.866);
  \coordinate (C2) at ( 1.5,  0.866);
  \coordinate (D2) at ( 2,    0);
  \coordinate (E2) at ( 1.5, -0.866);
  \coordinate (F2) at ( 0.5, -0.866);
  \coordinate (A3) at ( 1,    0);
  \coordinate (B3) at ( 1.5,  0.866);
  \coordinate (C3) at ( 2.5,  0.866);
  \coordinate (D3) at ( 3,    0);
  \coordinate (E3) at ( 2.5, -0.866);
  \coordinate (F3) at ( 1.5, -0.866);
  \newcommand{\edgeDot}[4][]{\draw[dotted]      (#2) -- (#3) node[midway,ideal,#1] {#4};}
  \newcommand{\edgeBold}[4][]{\draw[very thick] (#2) -- (#3) node[midway,ideal,#1] {#4};}
  \newcommand{\Zpq}{$\mathbb{Z}\!\bigl[\tfrac13\bigr]\!\bigl[\tfrac15\bigr]$}
  \edgeDot[above left]  {A1}{B1}{\Zpq}
  \edgeDot[above]       {B1}{C1}{\Zpq}
  \edgeDot[above right] {C1}{D1}{\Zpq}
  \edgeDot[right]       {D1}{E1}{\Zpq}
  \edgeDot[below]       {E1}{F1}{\Zpq}
  \edgeDot[below left]  {F1}{A1}{\Zpq}
  \edgeDot[above]       {A2}{B2}{\Zpq}
  \edgeDot[above]       {B2}{C2}{\Zpq}
  \edgeDot[above right] {C2}{D2}{\Zpq}
  \edgeDot[right]       {D2}{E2}{\Zpq}
  \edgeDot[below]       {E2}{F2}{\Zpq}
  \edgeDot[below]       {F2}{A2}{\Zpq}
  \edgeBold[above right]       {A3}{B3}{$\langle 7^1\rangle$}
  \edgeBold[above]       {B3}{C3}{$\langle 7^2\rangle$}
  \edgeBold[above right] {C3}{D3}{$\langle 7^3\rangle$}
  \edgeBold[right]       {D3}{E3}{$\langle 7^4\rangle$}
  \edgeBold[below]       {E3}{F3}{$\langle 7^5\rangle$}
  \edgeBold[below right]       {F3}{A3}{$\langle 7^6\rangle$}
  \foreach \v in {A1,B1,C1,D1,E1,F1,
                  A2,C2,D2,E2,
                  A3,C3,D3,E3,
                  B2,B3,F2,F3}
        \fill (\v) circle (1.4pt);
\end{tikzpicture}
\caption{After base change to $\Spec\left(\Z\!\left[\tfrac13\right]\!\left[\tfrac15\right]\right)$, only the rightmost hexagon imposes non-trivial conditions.}
\label{fig:basechange-hexagon}
\end{figure}

Note that the edge-labeled graph in Figure~\ref{fig:basechange-hexagon} is now over $\Z[\frac{1}{3}][\frac{1}{5}]$ and not the original ring. Notice that the edges labeled with $\Z[\frac{1}{3}][\frac{1}{5}]$ do not impose any stipulations on adjacent vertices. As a consequence the only non-trivial requirements fall along the vertices which form a cycle. The hexagon in this case is an example of a fully determining subgraph. If all edges impose a trivial GKM-condition after base change, then we say the graph has become trivial. In the example with $U = \Spec(\Z[\frac{1}{3}][\frac{1}{5}]) \subseteq \Spec(\Z)$, we write $G_U$ to denote the graph in Figure~\ref{fig:basechange-hexagon}.

We will now explore the geometry of trivializing an edge upon base change. Suppose we have a two-vertex edge-labeled graph with an associated edge diagram $R_u \rightarrow R/I_{uv} \leftarrow R_v$, and for simplicity assume the ideals are principal. Then if we take the complement $U = V(I_{uv})^c$ of $V(I_{uv})$ to obtain a standard open set, after base change we have the edge diagram $(R_U)_u \rightarrow 0 \leftarrow (R_U)_v$. The fiber product for this will be isomorphic to $R_U \times R_U$. At the level of schemes note that:
\begin{align*}
  \Spec(R_U \times R_U) &= \Spec(R_U) \bigsqcup \Spec(R_U)
\end{align*}

This corresponds to two copies of $\Spec(R_U)$ that are not glued together. Hence after restricting to the standard open set $U = V(I_{uv})^c$, the gluing instructions of the edge-labeled diagram become ``take two copies of $\Spec(R_U)$ and do not glue.'' This is because the prime ideals that were to be glued along were knocked out of the spectrum entirely. If the ideal were not principal, then the complement may not be a distinguished open set (or even an affine scheme). In this case, we could still cover the open set with distinguished opens, and then the trivialization claim would hold within each distinguished open.

Now consider the flat base change of $\Spec(R_{G})$ to a distinguished $U$. If we consider the following pullback diagram:

\[
\begin{tikzcd}[row sep=3.5em, column sep=4em]
\Spec(R_{G})\times_{\Spec(R)} U
  \arrow[r, hook, ""]
  \arrow[d, ""'] &
\Spec(R_{G}) \arrow[d, ""] \\[0.3em]
U \arrow[r, hook, ""'] &
\Spec(R)
\arrow[ul, phantom, very near start]
\end{tikzcd}
\]

The net effect is to restrict each copy of $\Spec(R)$ sitting inside $\Spec(R_{G})$ to $U$. This is visualized in Figure~\ref{fig:flat-base-change}. To see why this leads to the fact flat base change commutes with colimits, consider again the cycle in Figure~\ref{fig:cycle-pushout}. If we pullback the ring of splines along a standard open set with $\langle x \rangle$ knocked out, then the cycle comes unglued at $\langle x \rangle$. But if we flat base change before gluing, then we never glue along $\langle x \rangle$ in the first place. This is visualized in Figure~\ref{fig:flat-base-change-colimit}.

\begin{figure}
  \centering
  \includegraphics[width=0.95\linewidth]{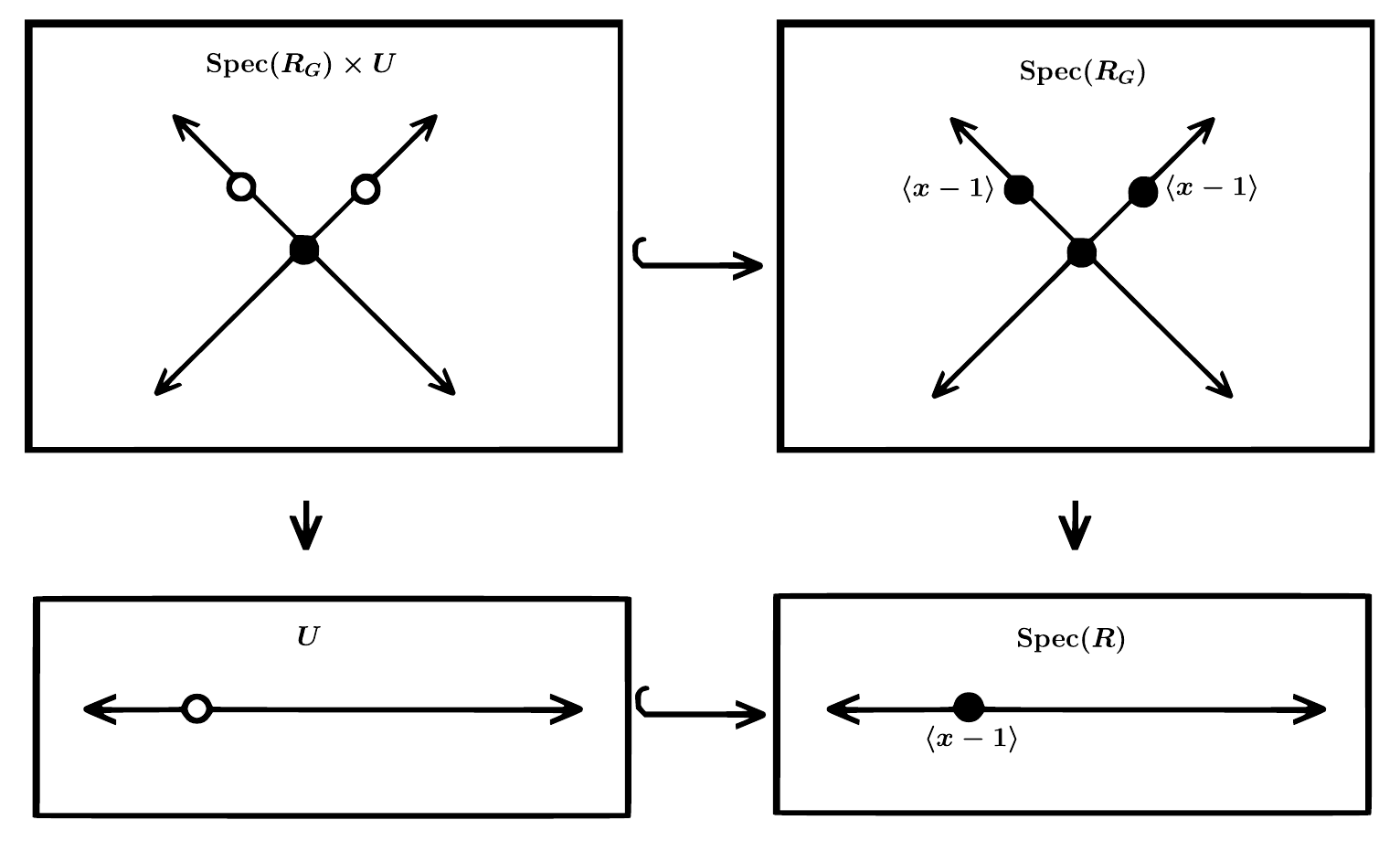}
  \caption{Flat base change diagram.}
  \label{fig:flat-base-change}
\end{figure}

\begin{figure}
  \centering
  \includegraphics[width=0.95\linewidth]{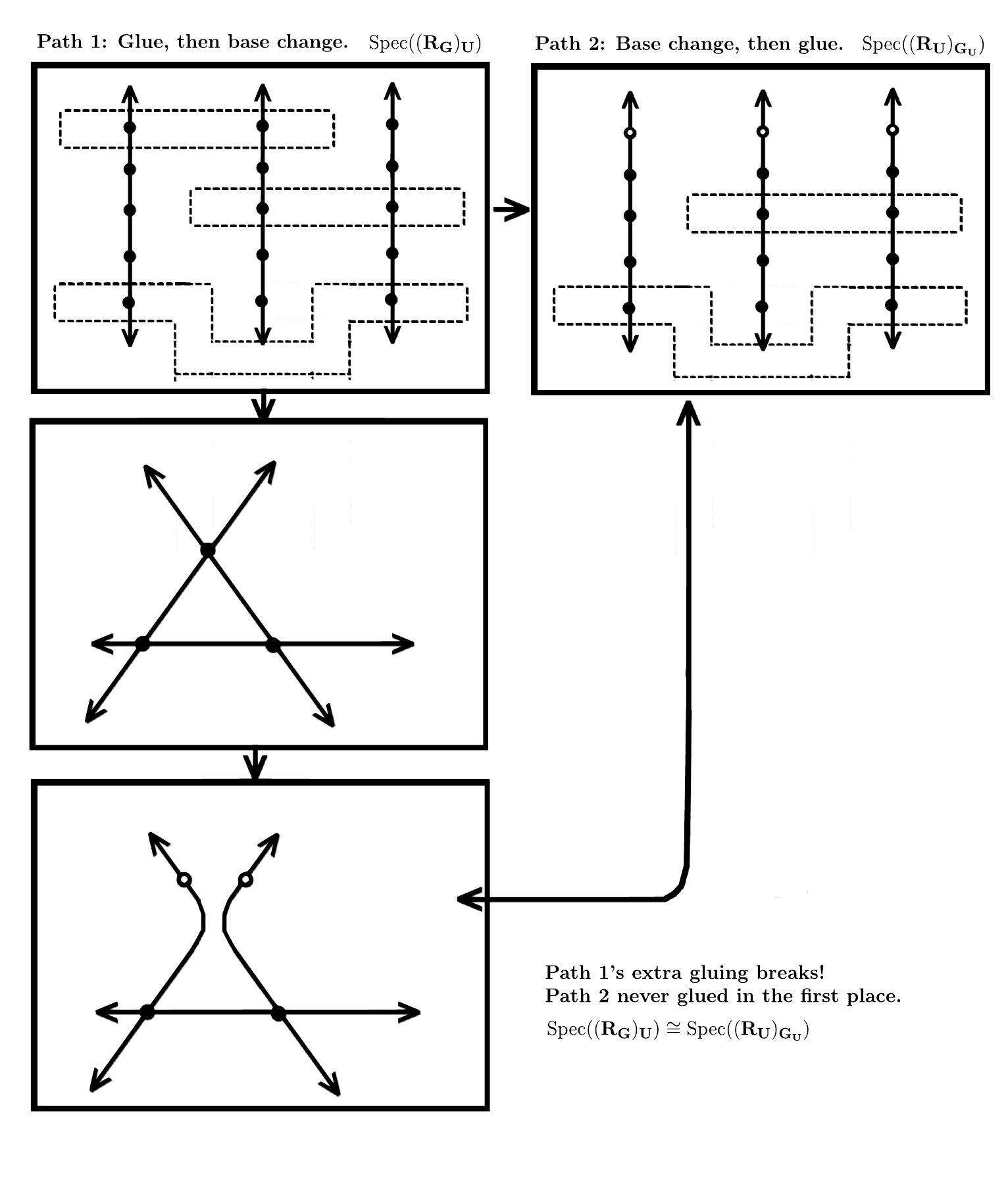}
  \caption{Flat base change commuting with colimit.}
  \label{fig:flat-base-change-colimit}
\end{figure}

In summary, in this section we first showed how the functor $F_S$ induces an edge-labeled graph $F_S(G)$ from an edge-labeled graph $G$. Next we used this to define graph restriction, and then showed that flat base change commutes with colimits of spline diagrams of schemes. That essentially establishes that flat base change induces a transformation of spline structure morphisms. Finally, we defined lifts of edge-labeled graphs along suitable functors, graph restriction, and then explored an example which motivated the notion of an  edge-labeled graph being locally determined by a subgraph. Thus, in the previous section we established the scheme-theoretic perspective, and now we have collected enough foundational results to work with it. We can now quickly harvest several results that do not exist in the generalized spline literature.

\section{Applications}
Before we begin with applications, it is worth mentioning the importance of standard (or distinguished) open sets versus general open sets as they relate to edge subschemes. In general, an open set of $\Spec(R)$ need not be affine, so flat base change to an arbitrary open subset could take $\Spec(R_G)$ to a scheme that is not affine, and we would lose access to the standard results for generalized splines. If we want to create a standard open set in which an edge ideal $I$ trivializes, we should not just work on the complement $U = V(I)^c$ of $V(I)$, but instead consider inverting generators of $I$. On the other hand, if $U$ is not standard, we could still cover $U$ by standard opens. But it is possible these standard opens could interact with the other edge labels in unexpected ways, so caution is needed for this route, as we cannot assume $I$ trivializes but not any of the other edge labels within the standard covering of $U$. Because of these complexities, the results which follow will either directly assume the cover is standard as a hypothesis, or construct standard open sets explicitly by inverting elements.

In \textit{multivariate generalized splines and syzygies on graphs} Samet Sar{\i}o\u{g}lan and Selma Alt{\i}nok established that if $R = k[x,y]$ is a polynomial ring over a field with edge-labeled graph $G$ a cycle, then $R_{G}$ is free \cite{sarioglan2024multivariate}. We can now obtain a generalization as follows:

\begin{corollary}
\label{cor:local-cycle-free}
Suppose $R = k[x,y]$ is a ring of polynomials in two variables over a field $k$. Let $G$ be an edge-labeled graph with edges labeled by principal ideals. If there is a distinguished open cover of $\Spec(R)$ such that $G$ is locally trivial or fully determined by cycles, then $R_G$ is free.
\end{corollary}
\begin{proof}
Let $R$, $G$, and $\{U_{f_i}\}$ be given as specified. Notice a cycle can clearly be lifted to a cycle over the original base ring, which is free by the criteria of Samet Sar{\i}o\u{g}lan and Selma Alt{\i}nok  \cite{sarioglan2024multivariate}, which implies $G$ is locally free by Corollary~\ref{cor:globally-free-implies-locally-free}. The claim follows by Corollary~\ref{cor:locally-free}.
\end{proof}

The previous result illustrates the one-two impact of Corollary~\ref{cor:globally-free-implies-locally-free} followed by Corollary~\ref{cor:locally-free}. They allow us to apply results about the original ring in the local context, and then patch the local pictures together into a global result. A concrete example appears in the appendix, where we prove freeness for a graph formed by interlinking hexagons. We can also obtain a more general form. For multivariate generalized splines over a polynomial ring in arbitrarily many variables, a cycle (or graph of disconnected cycles) is free provided the linear span of the principal edge-labelings (called the rank of the cycle) is at most $2$ \cite{sarioglan2024multivariate}.

\begin{corollary}
\label{cor:local-cycle-free2}
Suppose $R = k[x_1, \dots, x_n]$ is a ring of polynomials over a field $k$. Let $G$ be an edge-labeled graph with edges labeled by principal ideals. If there is a distinguished open cover of $\Spec(R)$ such that $G$ is locally trivial or fully determined by disconnected cycles with lifts of rank at most $2$ (which are known to be free) then $R_G$ is free.
\end{corollary}
\begin{proof}
Let $R$, $G$, and $R_G$ be given as specified. Since each fully determining local cycle lifts to one of rank $2$, they are free by the criteria of Samet Sar{\i}o\u{g}lan and Selma Alt{\i}nok  \cite{sarioglan2024multivariate}, which implies $G$ is locally free by Corollary~\ref{cor:globally-free-implies-locally-free}. The claim follows by Corollary~\ref{cor:locally-free}.
\end{proof}

These first applications relied on a standard cover. However, we have two immediate porisms that follow because localization at a prime ideal is also flat base change.

\begin{porism}
Suppose $R = k[x,y]$ is a ring of polynomials in two variables over a field $k$. Let $G$ be an edge-labeled graph with edges labeled by principal ideals. If for each prime ideal $\mathfrak p$ of $R$, $G_{\mathfrak p}$ is trivial or fully determined by cycles, then $R_G$ is free.
\end{porism}
\begin{porism}
Suppose $R = k[x_1, \dots, x_n]$ is a ring of polynomials over a field $k$. Let $G$ be an edge-labeled graph with edges labeled by principal ideals. If for each prime ideal $\mathfrak p$ of $R$, $G_{\mathfrak p}$ is trivial or fully determined by disconnected cycles with lifts of rank at most $2$ (which are known to be free) then $R_G$ is free.
\end{porism}
\begin{proof}
Both Corollary~\ref{cor:globally-free-implies-locally-free}, and Corollary~\ref{cor:locally-free} work with either a standard cover or localization at all prime ideals, so the previous proofs carry over verbatim.
\end{proof}

The proofs of the previous results imply that the problem of freeness for multivariate generalized splines may be encoded by local rings. We can now codify this with a theorem.

\begin{theorem}
Let $R= A[x_1,\dots,x_n]$ with $A$ a PID or field. Every edge-labeled graph with principal ideal labels is free over $R$ if and only if every edge-labeled graph with principal ideal labels is free over each $R_{\mathfrak p}$, with $\mathfrak p$ a prime ideal of $R$.
\end{theorem}
\begin{proof}
Let $R$ be given as specified, and let $\mathfrak p$ be some prime ideal of $R$. Suppose every edge-labeled graph over $R$ is free. Then for any edge labeled graph $G$ over $R_{\mathfrak p}$, we can lift to a corresponding edge-labeled graph over $R$ which is free by hypothesis. Since free-liftable implies free by Lemma~\ref{lma:free-lift}, $G$ is free. So we can see the first implication holds.

For the second implication, suppose each edge-labeled graph over $R_{\mathfrak p}$ is free. But that just says every edge-labeled graph over $R$ is locally free, and since $R$ is subject to the Quillen-Suslin theorem, the edge-labeled graph must be free. So the second implication holds. Since both implications hold, the claim follows.
\end{proof}

A consequence of the previous result is that freeness of multivariate generalized splines is controlled entirely by the presence or absence of local obstructions. If there are no local obstructions to freeness, then freeness holds globally. If there are local obstructions to freeness, then freeness does not hold globally. This can be studied with traditional combinatorial techniques applied to local rings.

Generalized splines have not historically permitted edges labeled with $R$. Therefore, the following definition only appeared in the very special case where all edge-labels were exactly the same.

\begin{definition} 
Let $R$ be a commutative ring and $G$ an edge-labeled graph. $G$ is said to be \textbf{mono} if each partially determining subgraph $H_i$ have all edges labeled by the same ideal $I_i$. If every $I_i$ has a minimum of $k$ generators, then $G$ is said to be \textbf{$k$-mono}. If every generator is regular (not a zero-divisor) then $G$ is said to be \textbf{regular $k$-mono}. If there exists a standard open cover $\{U_{f}\}$ of $R$ such that each $G_{f}$ is mono, or if $G_{\mathfrak p}$ is mono for each prime ideal $\mathfrak p$ of $R$, then $G$ is said to be \textbf{locally mono}. The definitions for locally $k$-mono, locally regular $k$-mono, etc., are similar. Note that the graph where all edges are labeled with $R$ is \textbf{vacuously} regular $k$-mono.
\end{definition}

The next result will be used to preclude global freeness by a local argument.

\begin{lemma}
\label{lma:proj-ideal-principal}
Let $R$ be an integral domain. Suppose $\Pic(R)=0$. Then every non-zero projective ideal $I$ of $R$ is regular and principal.
\end{lemma}
\begin{proof}
Let $I \subseteq R$ be an ideal which is projective as an $R$-module. Since $I$ is finitely generated and projective, it is locally free: for each prime $\mathfrak p$, $I_{\mathfrak p}$ is a free $R_{\mathfrak p}$-module. Because $I_{\mathfrak p} \subseteq R_{\mathfrak p}$ and $R_{\mathfrak p}$ is free of rank $1$, $I_{\mathfrak p}$ must be free of rank $\le 1$. Since $R$ is an integral domain, $I_{\mathfrak p}$ is free of rank $1$. Thus $I$ is a locally free rank $1$ $R$-module. Now since $\Pic(R) = 0$, every locally free rank $1$ $R$-module is isomorphic to $R$, which is only possible if $I$ is principal and its generator is regular.
\end{proof}

\begin{lemma}
\label{lma:mono}
Let $R$ be a commutative ring and let $G$ be an edge-labeled graph. If $G$ is fully determined by regular $1$-mono subgraphs, then $R_G$ is free. If $R$ is an integral domain, $\Pic(R) = 0$, and $G$ is partially determined by a $k>1$-mono subgraph, then $R_G$ is not free.  
\end{lemma}

\begin{proof}
Suppose $G$ is fully determined by regular $1$-mono subgraphs $H_i$. Since each $H_i$ has the same label $I_i$, we have the module isomorphism $R_{H_i} \cong R \oplus I_i \cdots \oplus I_i$ given by mapping the spline $(r, r+i_1, \dots , r+i_{n-1})$ to $(r, i_1, \dots , i_{n-1})$. Now since $I_i$ is principal and generated by a non-zero divisor, we know each $I_i$ is isomorphic to $R$ as an $R$-module. But that implies each $R_{H_i}$ is a sum of free modules, and hence free. Now since $R_G$ is fully determined by the $H_i$, we can see $R_G$ is a sum of free modules by Theorem~\ref{thm:summands} and hence free.

On the other hand, suppose $R$ is also an integral domain with $\Pic(R)=0$, and that $R_G$ is partially determined by $H$ which is a $k>1$-mono subgraph. As before, we still have that $R_H$ has $I$ as a summand, which implies $R_G$ has $I$ as a summand by Theorem~\ref{thm:summands}. For the sake of contradiction, assume $R_G$ is free. Then $I$ is a summand of a free module (Porism~\ref{por:summand}) hence projective. Now by Lemma~\ref{lma:proj-ideal-principal} it must be that $I$ is principal and generated by a single regular element of $R$. But that contradicts the fact $H$ was assumed to be $k>1$-mono or non-regular $1$-mono! So it must be that the assumption was wrong, and therefore $R_G$ is not free.
\end{proof}

The previous result was global in nature. Historically, such graphs would only show up when all edge-labels were exactly the same. Now, however, these types of graphs can show up locally, so they can arise in the study of graphs that are not mono. We can now extract the local-global version of Lemma~\ref{lma:mono} which follow from Corollary~\ref{cor:locally-free} and Corollary~\ref{cor:globally-free-implies-locally-free}.

\begin{corollary}
\label{cor:mono-not-free-cor}
Let $R$ be an integral domain, $G$ an edge-labeled graph, and $R_G$ the ring of splines. If there exists a standard open set $U_f$ such that $\Pic(R_f)=0$ and where $G_f$ is partially determined by a $k>1$-mono subgraph, then $R_G$ is not free. In particular, when $R$ is a UFD the integral domain and $\Pic(R_f)=0$ requirements are automatically satisfied.
\end{corollary}
\begin{proof}
Let $R$, $G$, $R_G$, and $U_f$ be given as specified. Since $G_f$ is partially determined by a $k>1$-mono subgraph, $R_f$ is an integral domain, and $\Pic(R_f)=0$, it follows that $G_f$ not free by Lemma~\ref{lma:mono}, and therefore $R_G$ is not free by Corollary~\ref{cor:globally-free-implies-locally-free}.
\end{proof}

The previous result requires $R$ to be an integral domain and also that $\Pic(R_f)=0$, which is automatically true for any UFD. Thus, for a ring of polynomials over a UFD, it suffices to find $U_f$ with a $k>1$-mono subgraph to satisfy the hypothesis. In the appendix, there is a table of non-free edge-labeled graphs over polynomial rings that were selected via this observation.

\begin{corollary}
\label{cor:mono-free-cor}
Suppose $R$ is a commutative ring and let $G$ be locally regular $1$-mono, then $R_G$ is projective. If $R$ is a PID or of the form $A[x_1,\dots,x_n]$ with $A$ a PID or field, then $R_G$ is also free.
\end{corollary}
\begin{proof}
The conditions imply $G$ is locally free because of Lemma~\ref{lma:mono}, so the claim follows by Corollary~\ref{cor:locally-free}.
\end{proof}

\begin{corollary}
Suppose $R$ is an integral domain, $G$ is an edge-labeled graph with labels principal ideals, and $R_G$ the ring of splines. If every pair of distinct edge subschemes of $G$ is pairwise disjoint, then $R_G$ is free.
\end{corollary}
\begin{proof}
Let $\mathfrak p$ be any prime ideal of $R$. Notice $G_{\mathfrak p}$ has trivial edges for all except the edge subschemes which intersect $\mathfrak p$, of which there is at most one by the hypothesis, and furthermore the associated ideal is regular and principal. Therefore $G_{\mathfrak p}$ is either regular $1$-mono properly, or vacuously so provided $\mathfrak p$ does not intersect any of the edge subschemes. The claim follows by Corollary~\ref{cor:locally-free}.
\end{proof}

Consider the following which is close to a well known result in the generalized spline literature, which we can now (almost) reprove using Corollary~\ref{cor:mono-free-cor}. A citation to the original proof is not included because the author was unable to determine where the result first appeared.

\begin{corollary}
\label{cor:almost}
Let $R$ be a commutative ring, $G$ an edge-labeled graph, and $R_G$ the ring of splines. If $R$ is a principal ideal domain, and if none of the generators of labeling ideals has irreducible factors with multiplicity greater than $1$, then $R_G$ is free.
\end{corollary}
\begin{proof}
Let $r_1, r_2, \dots, r_n$ be the irreducible factors of the generators of the edge-labeled graph, each of which has multiplicity $1$ by hypothesis. Notice that when $\mathfrak p$ is a prime ideal not equal to $\langle r_1 \rangle, \langle r_2 \rangle, \dots, \langle r_n\rangle$, then $G_{\mathfrak p}$ is vacuously regular 1-mono because every edge is trivial. Otherwise, $G_{\langle r_i \rangle}$ has $\langle r_i \rangle$ as the only non-trivial edge-ideal, so $G_{\langle r_i \rangle}$ is regular $1$-mono. So we can see $G$ is locally regular $1$-mono and the claim follows by Corollary~\ref{cor:mono-free-cor}
\end{proof}

The previous proof relies on the fact the Zariski topology of a PID is relatively simple, which prevents irreducible generators from intersecting within the Zariski topology. The result fails to reprove the full result that a ring of splines over a PID is free because there are currently no results on local rings to pull on. As an example of the problem of multiplicity, suppose we have an edge-labeled graph with edges $\langle 2 \rangle$ and $\langle 4 \rangle$. Then localizing at $\langle 2 \rangle$ allows both of these distinct edge-labels to survive, so we lose access to the regular $1$-mono argument. 

\begin{proposition}
\label{prop:boundary-trivialization}
Let $R$ be a UFD, $G$ an edge-labeled graph, and $R_G$ the ring of splines. Suppose $G$ contains (but is not necessarily partially determined by) a subgraph $H$. Let $U_{\partial H}$ be the open set obtained by inverting the multiplicative set formed by the generators of the $I_e \in \partial H$. If $H_{U_{\partial H}}$ is $k>1$-mono, then $R_G$ is not free.
\end{proposition}
\begin{proof}
Let $R$, $G$, $R_G$, $H$, and $U_{\partial H}$ be given as specified. Notice that each boundary ideal $I_e \in \partial H$ is trivial within $U_{\partial H}$. Therefore $G_{U_{\partial H}}$ is partially determined by $H_{U_{\partial H}}$, which is $k>1$-mono by hypothesis, and since $R$ is a UFD the claim follows by Corollary~\ref{cor:mono-not-free-cor}
\end{proof}

The proposition implies that the presence of ideals with multiple generators has a strong chance of preventing an edge-labeled graph from being free, at least for Noetherian rings. This is because if the associated edge subscheme does not intersect any of the adjacent edge subschemes, then freeness is impossible. Thus, an ideal with more than $1$ generator in a Noetherian domain can only live in an edge-labeled graph with a free module of splines when it satisfies a very specific structural relationship with every adjacent ideal.

\begin{proposition}
\label{prop:noetherian-free}
Let $R$ be a Noetherian integral domain and let $G$ be an edge-labeled graph where every edge label is principal. Then if for every pair of edge labels $I_i$ and $I_j$ we have that $I_i + I_j = R$ (i.e. every edge scheme $\Spec(R/I)$ is pairwise disjoint) then $R_G$ is projective. If $R$ is a PID or of the form $A[x_1,\dots,x_n]$ with $A$ a PID or field, then $R_G$ is also free.
\end{proposition}
\begin{proof}
Let $U^{I_i}$ be the standard open obtained by inverting every generator except the one for $I_i$. Now since the edge subschemes are all pairwise disjoint, it must be that every point in $\Spec(R)$ is contained in at least one of the $U^{I_i}$, so they form an open covering. To see each $G_{U^{I_i}}$ is $1$-mono, notice that after inverting every generator except one, every edge trivializes except the one for the generator that wasn't inverted. So within each standard open set, the edge-labeled graph has all non-trivial edges labeled by $I_i$. Finally, because $R$ is integral by hypothesis, the generators are also regular, so $G$ is locally regular $1$-mono, and hence projective by Corollary~\ref{cor:mono-free-cor}. If $R=A[x_1,\dots,x_n]$ with $A$ a principal ideal domain or field, then $R_G$ is free by Quillen-Suslin.
\end{proof}

The previous result is heavy-handed: the criteria forces a single ideal to survive within each standard open set, when in reality multiple disconnected components could survive, provided the edges within each connected component share the same ideal. The appendix contains examples of the results in this section applied to control freeness of concrete edge-labeled graphs.

\section{Deletion and Contraction}
In the foundational paper \textit{Generalized Splines on Arbitrary Graphs} by Gilbert, Tymoczko, and Viel \cite{gilbert2016generalized}, Question 7.4 asks:

\begin{center}
\emph{``How do classical graph-theoretic constructions (like deletion and contraction) affect the algebraic structure of splines $R_{G}$?''}
\end{center}

We can now address the question of deletion and contraction in terms of how it affects $\Spec(R_{G})$. Deletion of an edge was already discussed previously via addition of an edge. It corresponds to passing from the coequalizer of the following diagram back to $\Spec(R_{G})$.
\[
\Spec(R_{G'}) \leftarrow \Spec(R_{G}) \;\leftleftarrows\; \Spec(R/I_{uv}).
\]

This is because to add an edge, we pass from $\Spec(R_{G})$ to $\Spec(R_{G'})$. It follows deletion of an edge reverses this process. This is not an operation easily inverted: given $\Spec(R_{G'})$ (with an added edge), we cannot just delete the copy of $\Spec(R/I_{uv})$ that was used to glue. To see why, notice that if we glue three schemes along the same shared point, then deletion of the shared point will disconnect all three schemes. However, if we delete an edge, this means two of the schemes are glued, while the third remains disconnected, so we obtain a scheme with two connected components. Therefore, any algebraic attempt to describe edge deletion from the view of $R_{G'}$ will run into significant difficulty. 

Vertex deletion can be understood as first deleting all edges to the target vertex. Once the target vertex has no edges, it contributes an extra disconnected copy of $\Spec(R)$ to the spectrum. The last step is to remove the extra disconnected copy of $\Spec(R)$.

Finally it remains to discuss contraction. In the case of contraction we need to delete an edge and then identify two vertices. Since deletion of an edge was already described, we need only discuss identification of vertices. In this case, we are passing to and from the coequalizer of the diagram:

\[
\Spec(R_{G'}) \leftarrow \Spec(R_{G}) \;\leftleftarrows\; \Spec(R).
\]

Notice the simple description of how $\Spec(R_{G})$ is impacted doesn't immediately describe the algebraic structure of the ring of global sections. The fact the algebraic side is inscrutable is analogous to the tensor product of vector bundles: describing tensor product of vector bundles is simple algebraically, since it corresponds to fiber-wise tensor product. Describing the geometry of the resulting vector bundle, however, is not such an easy task. Similarly, we've described how $\Spec(R_{G})$ is impacted by deletion and contraction, without directly addressing how the sections are impacted.

\section{Conclusion}
We conclude with a few directions for future study. 
\begin{itemize}
\item Edge-labeled graphs over local rings are not well studied. However, they have clear application to a wide variety of rings due to their role in the local--global apparatus. 
\item How do spline structure morphisms behave within flat families?
\item Spline diagrams can appear in other categories, so what do they tell us in those contexts? For example, some definitions of Cartier divisors have a feature which resembles the GKM-condition.
\item How can we leverage schemes to study $r$-smoothness in the context of generalized splines? 
\item What does the geometry of $\Spec(R_{G})$ as it relates to $\Spec(R)$ tell us about $R_{G}$ itself?
\item The possible criteria to force open covers where an edge-labeled graph degenerates into a desired form is likely far from exhausted, so expanding the known criteria is a clear direction for future inquiry. 
\end{itemize}

\section{Acknowledgements}
The author would like to acknowledge Luke Smith, who introduced him to generalized splines minutes before a shared complex analysis class began. The author would also like to acknowledge Scott LaLonde, who advised the author's undergraduate senior seminar project that featured an early form of this paper, and Katie Anders and Michael DiPasquale for helpful feedback. Finally, the author would like to acknowledge his brother, Cody Stoltz, who created many of the graphics which appear in this paper based on the author's hand-drawn sketches.

\providecommand{\bysame}{\leavevmode\hbox to3em{\hrulefill}\thinspace}
\providecommand{\MR}{\relax\ifhmode\unskip\space\fi MR }
\providecommand{\MRhref}[2]{%
  \href{http://www.ams.org/mathscinet-getitem?mr=#1}{#2}
}
\providecommand{\href}[2]{#2}

\section*{Appendix: Examples and Computational Experiments}
We now consider multiple examples and computational experiments. We begin with an example as follows. Consider the edge-labeled graph $G$ with ideals in $\R[x,y]$ in Figure~\ref{fig:three-hexagons-rxy}; it consists of three interlocked hexagons, no two sharing an edge (only vertices) where the first has principal ideals of the form $\langle (x - 3)^r \rangle$, the second $\langle (x - 5)^r \rangle$, and the third $\langle (x - 7)^r \rangle$. If we invert $x-3$, $x-5$, or $x-7$, we will trivialize the edge conditions for the corresponding hexagon with edge ideals of the form $\langle (x-a)^d \rangle$. Thus, we can selectively trivialize all but one hexagon at a time. The Zariski open cover that satisfies the conditions of the theorem is:
\begin{align*}
  U_1 &= \Spec((x-3, x-5)^{-1}\R[x,y]) \\
  U_2 &= \Spec((x-3, x-7)^{-1}\R[x,y]) \\
  U_3 &= \Spec((x-5, x-7)^{-1}\R[x,y])
\end{align*}

Where $(x-a, x-b)^{-1}\R[x,y]$ indicates we've inverted $x-a$ and $x-b$. And the restrictions $G_{U_1}$, $G_{U_2}$ and $G_{U_3}$ are shown in Figure~\ref{fig:restrict-U1}, Figure~\ref{fig:restrict-U2}, and Figure~\ref{fig:restrict-U3}. We can clearly see the graph is locally determined by a cycle along each standard open set, and since $\Spec(R[x,y])$ is covered it follows from Corollary~\ref{cor:local-cycle-free} that $\Spec(R[x,y])_G$ is free.

\begin{figure}
\centering
\begin{tikzpicture}[scale=2.4,
                    every node/.style={font=\scriptsize},
                    ideal/.style={inner sep=1pt}]
  \coordinate (A1) at (-1,0);
  \coordinate (B1) at (-0.5,  0.866);
  \coordinate (C1) at ( 0.5,  0.866);
  \coordinate (D1) at ( 1,    0);
  \coordinate (E1) at ( 0.5, -0.866);
  \coordinate (F1) at (-0.5, -0.866);
  \coordinate (A2) at ( 0,    0);
  \coordinate (B2) at ( 0.5,  0.866);
  \coordinate (C2) at ( 1.5,  0.866);
  \coordinate (D2) at ( 2,    0);
  \coordinate (E2) at ( 1.5, -0.866);
  \coordinate (F2) at ( 0.5, -0.866);
  \coordinate (A3) at ( 1,    0);
  \coordinate (B3) at ( 1.5,  0.866);
  \coordinate (C3) at ( 2.5,  0.866);
  \coordinate (D3) at ( 3,    0);
  \coordinate (E3) at ( 2.5, -0.866);
  \coordinate (F3) at ( 1.5, -0.866);
  \newcommand{\edge}[4][]{\draw (#2) -- (#3) node[midway,ideal,#1] {#4};}
  \edge[above left]  {A1}{B1}{$\langle (x-3)^1\rangle$}
  \edge[above]       {B1}{C1}{$\langle (x-3)^2\rangle$}
  \edge[below left]  {C1}{D1}{$\langle (x-3)^3\rangle$}
  \edge[above left]  {D1}{E1}{$\langle (x-3)^4\rangle$}
  \edge[below]       {E1}{F1}{$\langle (x-3)^5\rangle$}
  \edge[below left]  {F1}{A1}{$\langle (x-3)^6\rangle$}
  \edge[above]       {A2}{B2}{$\langle (x-5)^1\rangle$}
  \edge[above]       {B2}{C2}{$\langle (x-5)^2\rangle$}
  \edge[above right] {C2}{D2}{$\langle (x-5)^3\rangle$}
  \edge[right]       {D2}{E2}{$\langle (x-5)^4\rangle$}
  \edge[below]       {E2}{F2}{$\langle (x-5)^5\rangle$}
  \edge[below]       {F2}{A2}{$\langle (x-5)^6\rangle$}
  \edge[above]       {A3}{B3}{$\langle (x-7)^1\rangle$}
  \edge[above]       {B3}{C3}{$\langle (x-7)^2\rangle$}
  \edge[above right] {C3}{D3}{$\langle (x-7)^3\rangle$}
  \edge[right]       {D3}{E3}{$\langle (x-7)^4\rangle$}
  \edge[below]       {E3}{F3}{$\langle (x-7)^5\rangle$}
  \edge[below]       {F3}{A3}{$\langle (x-7)^6\rangle$}
  \foreach \v in {A1,B1,C1,D1,E1,F1,
                  A2,C2,D2,E2,
                  A3,C3,D3,E3,
                  B2,B3,F2,F3}
      \fill (\v) circle (1.4pt);
\end{tikzpicture}
\caption{Three hexagons in $\R[x,y]$ with edge ideals given as powers of $x-3$, $x-5$, and $x-7$.}
\label{fig:three-hexagons-rxy}
\end{figure}
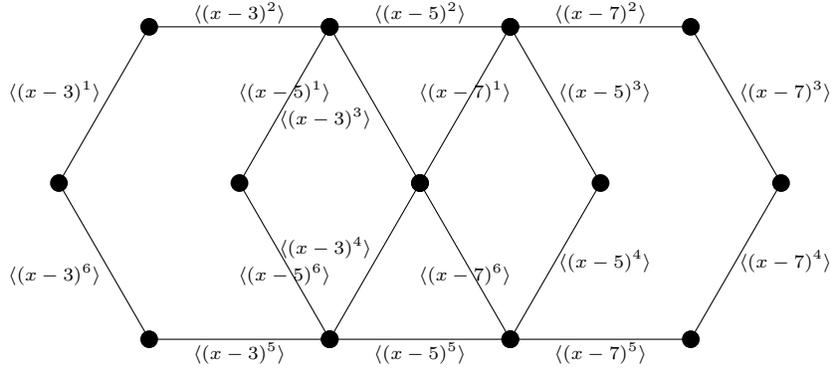

\begin{figure}
\centering 
\begin{tikzpicture}[scale=2.4,
                    every node/.style={font=\scriptsize},
                    ideal/.style={inner sep=1pt}]
  \coordinate (A1) at (-1,0);
  \coordinate (B1) at (-0.5,  0.866);
  \coordinate (C1) at ( 0.5,  0.866);
  \coordinate (D1) at ( 1,    0);
  \coordinate (E1) at ( 0.5, -0.866);
  \coordinate (F1) at (-0.5, -0.866);
  \coordinate (A2) at ( 0,    0);
  \coordinate (B2) at ( 0.5,  0.866);
  \coordinate (C2) at ( 1.5,  0.866);
  \coordinate (D2) at ( 2,    0);
  \coordinate (E2) at ( 1.5, -0.866);
  \coordinate (F2) at ( 0.5, -0.866);
  \coordinate (A3) at ( 1,    0);
  \coordinate (B3) at ( 1.5,  0.866);
  \coordinate (C3) at ( 2.5,  0.866);
  \coordinate (D3) at ( 3,    0);
  \coordinate (E3) at ( 2.5, -0.866);
  \coordinate (F3) at ( 1.5, -0.866);
  \newcommand{\edgeDot}[4][]{\draw[dotted]      (#2) -- (#3) node[midway,ideal,#1] {#4};}
  \newcommand{\edgeBold}[4][]{\draw[very thick] (#2) -- (#3) node[midway,ideal,#1] {#4};}
  \newcommand{\RxLocal}{$\mathcal O(U_1)$}
  \edgeDot[above left]  {A1}{B1}{\RxLocal}
  \edgeDot[above]       {B1}{C1}{\RxLocal}
  \edgeDot[above right] {C1}{D1}{\RxLocal}
  \edgeDot[right]       {D1}{E1}{\RxLocal}
  \edgeDot[below]       {E1}{F1}{\RxLocal}
  \edgeDot[below left]  {F1}{A1}{\RxLocal}
  \edgeDot[above]       {A2}{B2}{\RxLocal}
  \edgeDot[above]       {B2}{C2}{\RxLocal}
  \edgeDot[above right] {C2}{D2}{\RxLocal}
  \edgeDot[right]       {D2}{E2}{\RxLocal}
  \edgeDot[below]       {E2}{F2}{\RxLocal}
  \edgeDot[below]       {F2}{A2}{\RxLocal}
  \edgeBold[below right] {A3}{B3}{$\langle (x-7)^1\rangle$}
  \edgeBold[above]       {B3}{C3}{$\langle (x-7)^2\rangle$}
  \edgeBold[above right] {C3}{D3}{$\langle (x-7)^3\rangle$}
  \edgeBold[right]       {D3}{E3}{$\langle (x-7)^4\rangle$}
  \edgeBold[below]       {E3}{F3}{$\langle (x-7)^5\rangle$}
  \edgeBold[below]       {F3}{A3}{$\langle (x-7)^6\rangle$}
  \foreach \v in {A1,B1,C1,D1,E1,F1,
                  A2,C2,D2,E2,
                  A3,C3,D3,E3,
                  B2,B3,F2,F3}
        \fill (\v) circle (1.4pt);
\end{tikzpicture}
\caption{Restriction $G_{U_1}$ where only the $x=7$ hexagon remains non-trivial.}
\label{fig:restrict-U1}
\end{figure}
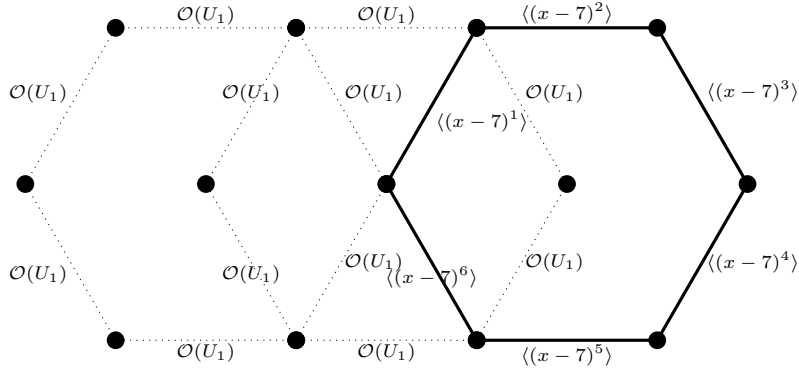
 
\begin{figure}
\centering
\begin{tikzpicture}[scale=2.4,
                    every node/.style={font=\scriptsize},
                    ideal/.style={inner sep=1pt}]
  \coordinate (A1) at (-1,0);  \coordinate (B1) at (-0.5, 0.866);
  \coordinate (C1) at ( 0.5,0.866);  \coordinate (D1) at (1,0);
  \coordinate (E1) at (0.5,-0.866); \coordinate (F1) at (-0.5,-0.866);
  \coordinate (A2) at (0,0);   \coordinate (B2) at (0.5,0.866);
  \coordinate (C2) at (1.5,0.866); \coordinate (D2) at (2,0);
  \coordinate (E2) at (1.5,-0.866); \coordinate (F2) at (0.5,-0.866);
  \coordinate (A3) at (1,0);   \coordinate (B3) at (1.5,0.866);
  \coordinate (C3) at (2.5,0.866); \coordinate (D3) at (3,0);
  \coordinate (E3) at (2.5,-0.866); \coordinate (F3) at (1.5,-0.866);
  \newcommand{\edgeDot}[4][]{\draw[dotted]      (#2)--(#3) node[midway,ideal,#1]{#4};}
  \newcommand{\edgeBold}[4][]{\draw[very thick] (#2)--(#3) node[midway,ideal,#1]{#4};}
  \newcommand{\OUtwo}{$\mathcal O(U_2)$}
  \edgeDot[above left]  {A1}{B1}{\OUtwo}
  \edgeDot[above]       {B1}{C1}{\OUtwo}
  \edgeDot[above right] {C1}{D1}{\OUtwo}
  \edgeDot[right]       {D1}{E1}{\OUtwo}
  \edgeDot[below]       {E1}{F1}{\OUtwo}
  \edgeDot[below left]  {F1}{A1}{\OUtwo}
  \edgeBold[above]       {A2}{B2}{$\langle (x-5)^1\rangle$}
  \edgeBold[above]       {B2}{C2}{$\langle (x-5)^2\rangle$}
  \edgeBold[above right] {C2}{D2}{$\langle (x-5)^3\rangle$}
  \edgeBold[right]       {D2}{E2}{$\langle (x-5)^4\rangle$}
  \edgeBold[below]       {E2}{F2}{$\langle (x-5)^5\rangle$}
  \edgeBold[below]       {F2}{A2}{$\langle (x-5)^6\rangle$}
  \edgeDot[above]       {A3}{B3}{\OUtwo}
  \edgeDot[above]       {B3}{C3}{\OUtwo}
  \edgeDot[above right] {C3}{D3}{\OUtwo}
  \edgeDot[right]       {D3}{E3}{\OUtwo}
  \edgeDot[below]       {E3}{F3}{\OUtwo}
  \edgeDot[below]       {F3}{A3}{\OUtwo}
  \foreach \v in {A1,B1,C1,D1,E1,F1,
                  A2,C2,D2,E2,
                  A3,C3,D3,E3,
                  B2,B3,F2,F3}
        \fill (\v) circle (1.4pt);
\end{tikzpicture}
\caption{Restriction $G_{U_2}$ where only the $x=5$ hexagon remains non-trivial.}
\label{fig:restrict-U2}
\end{figure}
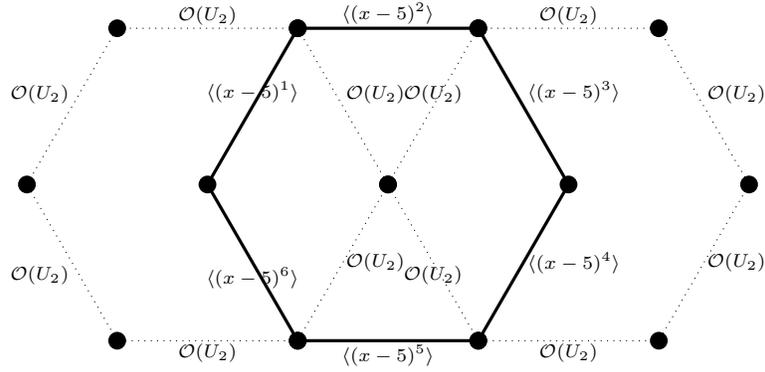
 
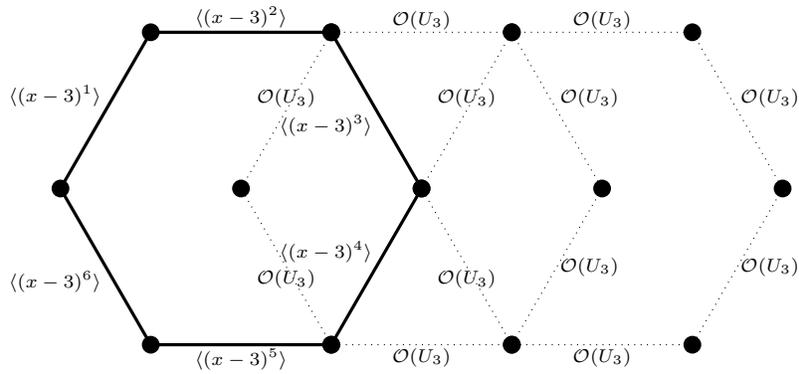
\begin{figure}
\centering
\begin{tikzpicture}[scale=2.4,
                    every node/.style={font=\scriptsize},
                    ideal/.style={inner sep=1pt}]
  \coordinate (A1) at (-1,0);  \coordinate (B1) at (-0.5, 0.866);
  \coordinate (C1) at ( 0.5,0.866);  \coordinate (D1) at (1,0);
  \coordinate (E1) at (0.5,-0.866); \coordinate (F1) at (-0.5,-0.866);
  \coordinate (A2) at (0,0);   \coordinate (B2) at (0.5,0.866);
  \coordinate (C2) at (1.5,0.866); \coordinate (D2) at (2,0);
  \coordinate (E2) at (1.5,-0.866); \coordinate (F2) at (0.5,-0.866);
  \coordinate (A3) at (1,0);   \coordinate (B3) at (1.5,0.866);
  \coordinate (C3) at (2.5,0.866); \coordinate (D3) at (3,0);
  \coordinate (E3) at (2.5,-0.866); \coordinate (F3) at (1.5,-0.866);
  \newcommand{\edgeDot}[4][]{\draw[dotted]      (#2)--(#3) node[midway,ideal,#1]{#4};}
  \newcommand{\edgeBold}[4][]{\draw[very thick] (#2)--(#3) node[midway,ideal,#1]{#4};}
  \newcommand{\OUthree}{$\mathcal O(U_3)$}
  \edgeBold[above left]  {A1}{B1}{$\langle (x-3)^1\rangle$}
  \edgeBold[above]       {B1}{C1}{$\langle (x-3)^2\rangle$}
  \edgeBold[below left]  {C1}{D1}{$\langle (x-3)^3\rangle$}
  \edgeBold[above left]  {D1}{E1}{$\langle (x-3)^4\rangle$}
  \edgeBold[below]       {E1}{F1}{$\langle (x-3)^5\rangle$}
  \edgeBold[below left]  {F1}{A1}{$\langle (x-3)^6\rangle$}
  \edgeDot[above]       {A2}{B2}{\OUthree}
  \edgeDot[above]       {B2}{C2}{\OUthree}
  \edgeDot[above right] {C2}{D2}{\OUthree}
  \edgeDot[right]       {D2}{E2}{\OUthree}
  \edgeDot[below]       {E2}{F2}{\OUthree} 
  \edgeDot[below]       {F2}{A2}{\OUthree} 
  \edgeDot[above]       {A3}{B3}{\OUthree}
  \edgeDot[above]       {B3}{C3}{\OUthree}
  \edgeDot[above right] {C3}{D3}{\OUthree}
  \edgeDot[right]       {D3}{E3}{\OUthree}
  \edgeDot[below]       {E3}{F3}{\OUthree}
  \edgeDot[below]       {F3}{A3}{\OUthree}
  \foreach \v in {A1,B1,C1,D1,E1,F1,
                  A2,C2,D2,E2,
                  A3,C3,D3,E3,
                  B2,B3,F2,F3}
        \fill (\v) circle (1.4pt);
\end{tikzpicture}
\caption{Restriction $G_{U_3}$ where only the $x=3$ hexagon remains non-trivial.}
\label{fig:restrict-U3}
\end{figure}

Before proceeding, we will discuss pathological behavior which requires caution to avoid. For example, $\Spec(\Q[x])$ includes points such as $\langle x^2-1 \rangle$, which means $V(x^2+1)$ is not empty, even though the solution set of $x^2+1$ over the rationals (and even reals) is empty. Thus, if we select an edge-labeled graph over $\Q[x,y]$ with generators that do not even intersect as real curves, there could still be non-trivial intersections within the Zariski topology, and thus, the covering condition could be violated unexpectedly. Proposition~\ref{prop:noetherian-free} used the check $I_i+I_j=R$ (instead of a naive geometric appeal to the Euclidean topology of plane curves) precisely to avoid this trap. That is because $I_i+I_j$ corresponds scheme theoretically to intersection, and when the intersection is empty the ideal sum contains a unit. 

Suppose we have $R=k[x,y]$ with $G$ a two vertex one edge graph labeled by $\langle x-1, y-2 \rangle$. Then we have two standard open sets, one given by inverting $x-1$ which trivializes $\langle x-1, y-2 \rangle$ and one given by inverting $y-2$ which also trivializes $\langle x-1, y-2 \rangle$. Within each standard open set, the module of splines becomes trivial. However, these standard open sets do not cover $\Spec(R)$ because both closed sets used to obtain them share a common point of intersection (namely $\langle x-1, y-2 \rangle$ itself), so the covering condition is not met.

Consider an edge-labeled graph with two vertices and two edges over $\R[x,y]$, one labeled by $\langle x^2+y^2-1 \rangle$ and one labeled by $\langle x-1, y \rangle$. If we obtain a standard open set by inverting $x^2+y^2-1$ it may be tempting to think the principally generated edge trivializes leaving the graph to be locally determined by a regular $2$-mono edge-labeled graph labeled with $\langle x-1, y \rangle$. However, when $x^2+y^2-1$ gets inverted, the only point of intersection between the two generators is knocked out of the prime spectrum, forcing $\langle x-1, y \rangle$ to trivialize. Before we continue with examples, we need propositions to reliably produce distinguished open covers.

\begin{proposition}
Let $R$ be a commutative ring and let $a$ be any element of $R$. Let $b = au + 1$ and $c = 1 + abv$, where $u, v \in R$. Then $V(a)$, $V(b)$ and $V(c)$ are pairwise disjoint, hence $\{U_{ab}, U_{ac}, U_{bc}\}$ is a cover of $\Spec(R)$. 
\end{proposition}
\begin{proof}
Notice $(-u)a + (1)b = 1$, $(-bv)a + (1)c = 1$, and $(-av)b + c = 1$. It follows that $\langle a, b \rangle = \langle 1 \rangle$, $\langle a, c \rangle = \langle 1 \rangle$, and $\langle b, c \rangle = \langle 1 \rangle$. The claim follows.
\end{proof}

\begin{proposition}
\label{prop:coveringR}
Let $R$ be any commutative ring with unity. Let $a_1, a_2, \dots, a_{n_1} \in R$ be regular elements. For any $r_1 \in R$, let $r_2 = r_1a_1+1, r_3 = r_2a_2 + 1, \dots, r_n = r_{n-1}a_{n-1} + 1$. The distinguished open sets $U_{r_i}$ cover $\Spec(R)$.
\end{proposition}
\begin{proof}
In order for $\{U_{r_i}\}$ to cover $\Spec(R)$, it is enough to show there is no point in every $V(r_i)$. To this end, notice that $r_2 - a_1r_1 = a_1r_1 + 1 - a_1r_1 = 1$, and for each $i, i+1$ we obtain a similar result. So it must be that $\langle r_i, r_{i+1} \rangle = \langle 1 \rangle$, and therefore the intersection every $V(r_i)$ is empty. That $\{U_{r_i}\}$ forms an open cover follows.
\end{proof}

The previous result is intuitively clear from a geometric point of view, since if for any $\mathfrak p \in \Spec(R)$ we have $r \equiv 0 \pmod{\mathfrak p}$, then $ar+1 \equiv 1 \pmod {\mathfrak p}$ for any $a$. A consequence of Proposition~\ref{prop:coveringR} is that if we have an edge-labeled graph with principal ideals that are labeled by combinations of exactly two such regular elements, then the module of splines is free. This is because any such edge-labeled graph is locally regular $1$-mono, since $U_{r_1}$ sets $r_1=1$ from the point of view of the generators, and $U_{r_2}$ sets $r_2=1$ from the point of view of the generators. That is, if $r_i$ is a unit, then $\langle r_i \rangle = \langle 1 \rangle$ within $U_{r_i}$.

We conclude with six edge-labeled graphs, each of which exemplifies a pattern of graph which is free when the specified covering conditions are met. Each of these were tested in Macaulay2 using the preceding propositions to generate suitable covers. After the graphs is a table of polynomials used as labels, followed by the results of the tests run in Macaulay2. The selection of polynomials and Macaulay2 code were created with assistance from GPT-5 according to the parameters specified by the author. The code used is included in the accompanying .m2 file for reproducibility.

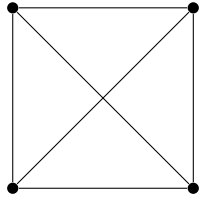
\begin{figure}[ht]
  \centering
  \begin{tikzpicture}[scale=1.2, every node/.style={circle, fill=black, inner sep=1.5pt}]
    \node (A) at (0,0) {};
    \node (B) at (2,0) {};
    \node (C) at (2,2) {};
    \node (D) at (0,2) {};
    
    \draw (A)--(B)--(C)--(D)--(A);
    \draw (A)--(C);
    \draw (B)--(D);
  \end{tikzpicture}
  \caption{Square graph with both diagonals filled in. Free for any pairwise disjoint standard open cover with principal ideals selected from the distinguished (regular) elements. If we select any pairwise disjoint distinguished open cover, and label the edges with products of these to any power, the result is locally regular $1$-mono.}
  \label{fig:square-diagonals}
\end{figure}

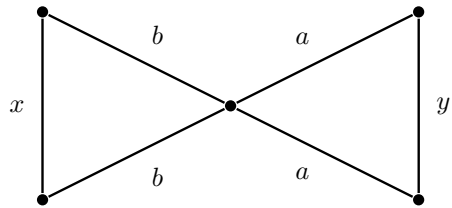
\begin{figure}[ht]
  \centering
  \begin{tikzpicture}[scale=1.25, line width=0.9pt]
    \tikzset{vertex/.style={circle, fill=black, inner sep=1.5pt}}

    \node[vertex] (Ltop) at (-2, 1) {};
    \node[vertex] (Lbot) at (-2,-1) {};
    \node[vertex] (Rtop) at ( 2, 1) {};
    \node[vertex] (Rbot) at ( 2,-1) {};
    \node[vertex] (M)    at ( 0, 0) {};

    \draw (Ltop)--(Lbot) node[midway, left=3pt] {$x$};
    \draw (Ltop)--(M)    node[midway, above right=2pt] {$b$};
    \draw (Lbot)--(M)    node[midway, below right=2pt] {$b$};

    \draw (Rtop)--(Rbot) node[midway, right=3pt] {$y$};
    \draw (Rtop)--(M)    node[midway, above left=2pt] {$a$};
    \draw (Rbot)--(M)    node[midway, below left=2pt] {$a$};
  \end{tikzpicture}
  \caption{Bowtie graph. Free for any distinguished cover $\{U_a, U_b\}$ and any selection of $x, y$. For a polynomial ring in three or more variables over a field, in each open set we get a regular $1$-mono edge-labeled graph and a cycle that lifts to one of rank $2$, which is free. }
  \label{fig:bowtie}
\end{figure}

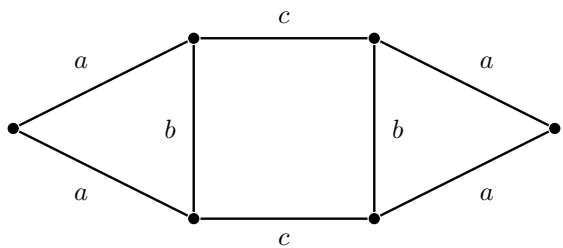
\begin{figure}[ht]
  \centering
  \begin{tikzpicture}[scale=1.2, line width=0.9pt]
    \tikzset{vertex/.style={circle, fill=black, inner sep=1.5pt}}

    \node[vertex] (LT) at (-1,  1) {}; 
    \node[vertex] (LB) at (-1, -1) {}; 
    \node[vertex] (RT) at ( 1,  1) {}; 
    \node[vertex] (RB) at ( 1, -1) {}; 

    \node[vertex] (Ltip) at (-3, 0) {};
    \node[vertex] (Rtip) at ( 3, 0) {};

    \draw (LT)--(RT) node[midway, above=2pt] {$c$};
    \draw (LB)--(RB) node[midway, below=2pt] {$c$};
    \draw (LT)--(LB) node[midway, left=3pt] {$b$};
    \draw (RT)--(RB) node[midway, right=3pt] {$b$};

    \draw (Ltip)--(LT) node[midway, above left=2pt] {$a$};
    \draw (Ltip)--(LB) node[midway, below left=2pt] {$a$};

    \draw (Rtip)--(RT) node[midway, above right=2pt] {$a$};
    \draw (Rtip)--(RB) node[midway, below right=2pt] {$a$};
  \end{tikzpicture}
  \caption{“Grilled Cheese” graph: Free for any standard open cover $\{U_a, U_b, U_c\}$. The grilled cheese graph degenerates into a rank $2$ cycle along $U_a$ (a square), into a rank $2$ cycle along $U_b$ (the boundary of the graph), and into two triangles of rank $2$ along $U_c$. }
  \label{fig:grilled-cheese}
\end{figure}

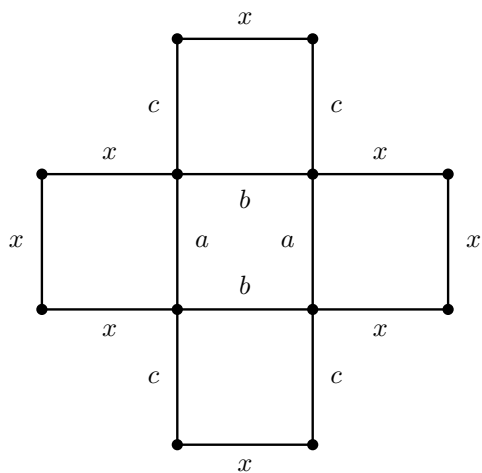
\begin{figure}[ht]
  \centering
  \begin{tikzpicture}[scale=0.9, line width=0.9pt]
    \tikzset{vertex/.style={circle, fill=black, inner sep=1.5pt}}

    \node[vertex] at (-1,  3) {}; \node[vertex] at ( 1,  3) {};
    \node[vertex] at (-1, -3) {}; \node[vertex] at ( 1, -3) {};
    \node[vertex] at (-3,  1) {}; \node[vertex] at (-3, -1) {};
    \node[vertex] at ( 3,  1) {}; \node[vertex] at ( 3, -1) {};
    \node[vertex] at (-1,  1) {}; \node[vertex] at ( 1,  1) {};
    \node[vertex] at (-1, -1) {}; \node[vertex] at ( 1, -1) {};

    \draw (-1,3) -- ( 1,3) node[midway, above=2pt] {$x$};          
    \draw (-1,1) -- ( 1,1) node[midway, below=2pt] {$b$};          
    \draw (-1,1) -- (-1,3) node[midway, left=3pt]  {$c$};          
    \draw ( 1,1) -- ( 1,3) node[midway, right=3pt] {$c$};          

    \draw (-1,-3) -- ( 1,-3) node[midway, below=2pt] {$x$};        
    \draw (-1,-1) -- ( 1,-1) node[midway, above=2pt] {$b$};        
    \draw (-1,-3) -- (-1,-1) node[midway, left=3pt]  {$c$};        
    \draw ( 1,-3) -- ( 1,-1) node[midway, right=3pt] {$c$};        

    \draw (-3, 1) -- (-1, 1) node[midway, above=2pt] {$x$};        
    \draw (-3,-1) -- (-1,-1) node[midway, below=2pt] {$x$};        
    \draw (-3,-1) -- (-3, 1) node[midway, left=3pt] {$x$};         
    \draw (-1,-1) -- (-1, 1) node[midway, right=3pt] {$a$};        

    \draw ( 1, 1) -- ( 3, 1) node[midway, above=2pt] {$x$};        
    \draw ( 1,-1) -- ( 3,-1) node[midway, below=2pt] {$x$};        
    \draw ( 3,-1) -- ( 3, 1) node[midway, right=3pt] {$x$};        
    \draw ( 1,-1) -- ( 1, 1) node[midway, left=3pt] {$a$};         

  \end{tikzpicture}
  \caption{The Plus Graph: free for any open cover $\{U_{ab}, U_{ac}, U_{bc}\}$ and selection of $x$. The plus graph requires an open set where two edges degenerate at a time. Along $U_{ab}$ we have the large rank $2$ cycle coming from the boundary. Along $U_{ac}$ we get a central rectangle cycle of rank $2$ and two regular $1$-mono edge-labeled graphs. Finally, along $U_{bc}$ we get two regular $1$-mono edge-labeled graphs and two rank $2$ squares. }
  \label{fig:plus-graph}
\end{figure}

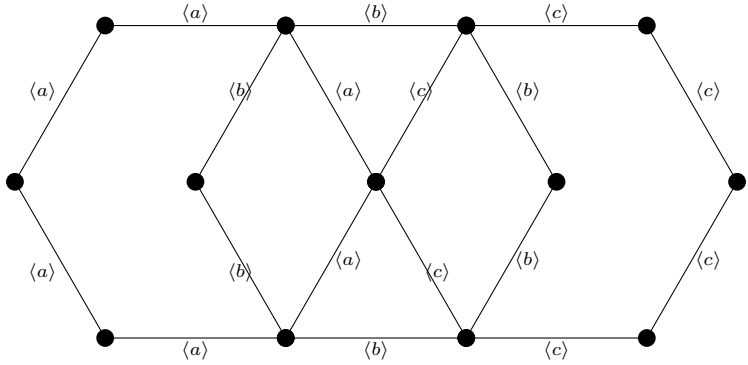
\begin{figure}
\centering
\begin{tikzpicture}[scale=2.4,
                    every node/.style={font=\scriptsize},
                    ideal/.style={inner sep=1pt}]
  \coordinate (A1) at (-1,0);
  \coordinate (B1) at (-0.5,  0.866);
  \coordinate (C1) at ( 0.5,  0.866);
  \coordinate (D1) at ( 1,    0);
  \coordinate (E1) at ( 0.5, -0.866);
  \coordinate (F1) at (-0.5, -0.866);

  \coordinate (A2) at ( 0,    0);
  \coordinate (B2) at ( 0.5,  0.866); 
  \coordinate (C2) at ( 1.5,  0.866);
  \coordinate (D2) at ( 2,    0);
  \coordinate (E2) at ( 1.5, -0.866);
  \coordinate (F2) at ( 0.5, -0.866); 

  \coordinate (A3) at ( 1,    0);     
  \coordinate (B3) at ( 1.5,  0.866); 
  \coordinate (C3) at ( 2.5,  0.866);
  \coordinate (D3) at ( 3,    0);
  \coordinate (E3) at ( 2.5, -0.866);
  \coordinate (F3) at ( 1.5, -0.866); 

  \newcommand{\edge}[4][]{\draw (#2) -- (#3) node[midway,ideal,#1] {#4};}

  \edge[above left]  {A1}{B1}{$\langle a\rangle$}
  \edge[above]       {B1}{C1}{$\langle a\rangle$}
  \edge[above right] {C1}{D1}{$\langle a\rangle$}
  \edge[right]       {D1}{E1}{$\langle a\rangle$}
  \edge[below]       {E1}{F1}{$\langle a\rangle$}
  \edge[below left]  {F1}{A1}{$\langle a\rangle$}

  \edge[above]       {A2}{B2}{$\langle b\rangle$}
  \edge[above]       {B2}{C2}{$\langle b\rangle$}
  \edge[above right] {C2}{D2}{$\langle b\rangle$}
  \edge[right]       {D2}{E2}{$\langle b\rangle$}
  \edge[below]       {E2}{F2}{$\langle b\rangle$}
  \edge[below]       {F2}{A2}{$\langle b\rangle$}

  \edge[above]       {A3}{B3}{$\langle c\rangle$}
  \edge[above]       {B3}{C3}{$\langle c\rangle$}
  \edge[above right] {C3}{D3}{$\langle c\rangle$}
  \edge[right]       {D3}{E3}{$\langle c\rangle$}
  \edge[below]       {E3}{F3}{$\langle c\rangle$}
  \edge[below right] {F3}{A3}{$\langle c\rangle$}

  \foreach \v in {A1,B1,C1,D1,E1,F1,
                  A2,C2,D2,E2,
                  A3,C3,D3,E3,
                  B2,B3,F2,F3}
      \fill (\v) circle (1.4pt);
\end{tikzpicture}
\caption{Three interlinking hexagons with perimeter labels $\langle a\rangle$, $\langle b\rangle$, and $\langle c\rangle$; no radial center links. The module of splines is free provided there is a regular $1$-mono covering given by $\{U_{ab}, U_{ac}, U_{bc}\}$.}
\label{fig:three-hexagons-no-radials}
\end{figure}

\begin{figure}
\centering
\begin{tikzpicture}[scale=2.4,
                    every node/.style={font=\scriptsize},
                    ideal/.style={inner sep=1pt}]
  \coordinate (A1) at (-1,0);
  \coordinate (B1) at (-0.5,  0.866);
  \coordinate (C1) at ( 0.5,  0.866);
  \coordinate (D1) at ( 1,    0);
  \coordinate (E1) at ( 0.5, -0.866);
  \coordinate (F1) at (-0.5, -0.866);

  \coordinate (A2) at ( 0,    0);
  \coordinate (B2) at ( 0.5,  0.866);
  \coordinate (C2) at ( 1.5,  0.866);
  \coordinate (D2) at ( 2,    0);
  \coordinate (E2) at ( 1.5, -0.866);
  \coordinate (F2) at ( 0.5, -0.866);

  \coordinate (A3) at ( 1,    0);
  \coordinate (B3) at ( 1.5,  0.866);
  \coordinate (C3) at ( 2.5,  0.866);
  \coordinate (D3) at ( 3,    0);
  \coordinate (E3) at ( 2.5, -0.866);
  \coordinate (F3) at ( 1.5, -0.866);

  \newcommand{\edge}[4][]{\draw (#2) -- (#3) node[midway,ideal,#1] {#4};}

  \edge[above left]  {A1}{B1}{$\langle a\rangle$}
  \edge[above]       {B1}{C1}{$\langle a\rangle$}
  \edge[above right] {C1}{D1}{$\langle a\rangle$}
  \edge[right]       {D1}{E1}{$\langle a\rangle$}
  \edge[below]       {E1}{F1}{$\langle a\rangle$}
  \edge[below left]  {F1}{A1}{$\langle a\rangle$}

  \edge[above]       {A2}{B2}{$\langle b\rangle$}
  \edge[above]       {B2}{C2}{$\langle b\rangle$}
  \edge[above right] {C2}{D2}{$\langle b\rangle$}
  \edge[right]       {D2}{E2}{$\langle b\rangle$}
  \edge[below]       {E2}{F2}{$\langle b\rangle$}
  \edge[below]       {F2}{A2}{$\langle b\rangle$}

  \edge[above right] {A3}{B3}{$\langle c\rangle$}
  \edge[above]       {B3}{C3}{$\langle c\rangle$}
  \edge[above right] {C3}{D3}{$\langle c\rangle$}
  \edge[right]       {D3}{E3}{$\langle c\rangle$}
  \edge[below]       {E3}{F3}{$\langle c\rangle$}
  \edge[below right] {F3}{A3}{$\langle c\rangle$}

  \edge {A1}{A2}{$\langle c\rangle$}
  \edge {B1}{A2}{$\langle c\rangle$}
  \edge {D1}{A2}{$\langle c\rangle$}
  \edge {F1}{A2}{$\langle c\rangle$}

  \edge {A3}{D2}{$\langle a\rangle$}
  \edge {C3}{D2}{$\langle a\rangle$}
  \edge {D3}{D2}{$\langle a\rangle$}
  \edge {E3}{D2}{$\langle a\rangle$}

  \foreach \v in {A1,B1,C1,D1,E1,F1, A2,B2,C2,D2,E2,F2, A3,B3,C3,D3,E3,F3}
      \fill (\v) circle (1.4pt);
\end{tikzpicture}
\caption{Perimeter labels have priority: left \(\langle a\rangle\), middle \(\langle b\rangle\), right \(\langle c\rangle\).
Only the non-perimeter center edges are added. Locally a rank $1$ cycle within any pairwise disjoint covering $\{U_{ab}, U_{ac}, U_{bc}\}$.}
\label{fig:hexagons-no-double-label}
\end{figure}
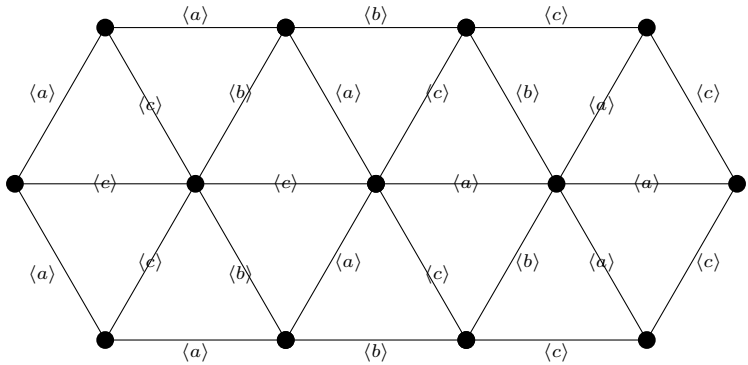

\providecommand{\Z}{\mathbb Z}
\providecommand{\Q}{\mathbb Q}
\providecommand{\Spec}{\operatorname{Spec}}
\newcolumntype{Y}{>{\raggedright\arraybackslash}X}

\begin{table}[ht]
\centering
\footnotesize
\setlength{\extrarowheight}{0.25em}
\begin{tabularx}{\linewidth}{|l|Y|Y|Y|}
\hline
\textbf{ID} &
\textbf{Distinguished elements} &
\textbf{Open cover / hypothesis used} &
\textbf{Graphs} \\
\hline\hline

PairCover-1 &
\(a=x,\; b=1-xy\). &
Two-open cover \(\{U_a,U_b\}\) with \(b=ar+1\). &
Bowtie \\
\hline

PairCover-2 &
\(a=y+z,\; b=1-(w+1)(y+z)\). &
Same \(\{U_a,U_b\}\) pattern with nontrivial \(a\). &
Bowtie \\
\hline

PairCover-3 &
\(a=zw+1,\; b=1-(x+y)(zw+1)\). &
Same construction; tests composite \(a\). &
Bowtie \\
\hline

PairCover-4 &
\(a=z+1,\; b=1-(x+y)(z+1)\). &
Shifted variable; robustness check. &
Bowtie \\
\hline\hline

TripleCover-T3-1 &
\(a=x,\; b=xy+1,\; c=(xy+1)z+1\). &
Three-open chain cover \(\{U_a,U_b,U_c\}\). &
GrilledCheese \\
\hline

TripleCover-T3-2 &
\(a=y-1,\; b=(y-1)(z+1)+1,\; c=bw+1\). &
Same chain with different seeds. &
GrilledCheese \\
\hline

TripleCover-T3-3 &
\(a=x+z,\; b=(x+z)(y+w)+1,\; c=b(x-1)+1\). &
Chain using nontrivial linear combinations. &
GrilledCheese \\
\hline

TripleCover-T3-4 &
\(a=x+z,\; b=(x+z)(y+w)+1,\; c=b(x+1)+1\). &
Perturbation of TripleCover-T3-3. &
GrilledCheese \\
\hline\hline

TripleCover-PD3-1 &
\(a=x,\; b=1-yx,\; c=1+zab\). &
Pairwise-disjoint triple; \(\{U_{ab},U_{ac},U_{bc}\}\). &
Plus,\; SquareDiag,\; LinkedHexagons \\
\hline

TripleCover-PD3-2 &
\(a=y,\; b=1-(z+w)y,\; c=1+(x+1)yb\). &
Same construction with different generators. &
Plus,\; SquareDiag,\; LinkedHexagons \\
\hline

TripleCover-PD3-3 &
\(a=z+1,\; b=1-x(z+1),\; c=1+(y+w)(z+1)b\). &
Shifted variable version. &
Plus,\; SquareDiag,\; LinkedHexagons \\
\hline

TripleCover-PD3-4 &
\(a=x+z,\; b=1-(y+1)(x+z),\; c=1+w(x+z)b\). &
Long-arm variant. &
Plus,\; SquareDiag,\; LinkedHexagons \\
\hline

\end{tabularx}
\caption{\textbf{Distinguished elements and covers.} All polynomials lie in \(k[x,y,z,w]\) (typically \(k=\Q\)). The PairCover-\(\#\) entries use a two-open cover \(\{U_a,U_b\}\). The TripleCover-T3-\(\#\) entries use an iterated three-open chain \(\{U_a,U_b,U_c\}\) with \(b=ar+1,\; c=bs+1\). The TripleCover-PD3-\(\#\) entries use \(\{U_{ab},U_{ac},U_{bc}\}\) with \(V(a),V(b),V(c)\) pairwise disjoint.}
\label{tab:structured-covers}
\end{table}

\begin{table}[ht]
\centering
\footnotesize
\setlength{\extrarowheight}{0.25em}
\begin{tabularx}{\linewidth}{|l|Y|Y|Y|}
\hline
\textbf{ID} &
\textbf{Polynomials / ideals} &
\textbf{Purpose} &
\textbf{Graphs} \\
\hline\hline

StressTest-1 &
\(a=xy+1,\; b=yz+1,\; c=zw+1\). &
No chain or pairwise-disjoint structure; cover not guaranteed. &
Bowtie,\; GrilledCheese,\; Plus,\; SquareDiag,\; LinkedHexagons \\
\hline

StressTest-2 &
\(a=x^2+y^2-5,\; b=(y-2)^2+z^2-7,\; c=z^2+w^2-3\). &
Curved vanishing sets that still meet in the Zariski topology. &
Bowtie,\; GrilledCheese,\; Plus,\; SquareDiag,\; LinkedHexagons \\
\hline

StressTest-3 &
\(a=x,\; b=y,\; c=z\). &
Coordinate axes; \(\langle a,b\rangle,\langle a,c\rangle,\langle b,c\rangle\) need not sum to \(R\). &
Bowtie,\; GrilledCheese,\; Plus,\; SquareDiag,\; LinkedHexagons \\
\hline

\end{tabularx}
\caption{\textbf{Stress tests.} These inputs deliberately violate the covering hypotheses. The experiments record rank and global freeness with no theoretical prediction.}
\label{tab:stress-tests}
\end{table}

\begin{table}[ht]
\centering
\footnotesize
\setlength{\extrarowheight}{0.25em}
\begin{tabularx}{\linewidth}{|l|Y|Y|}
\hline
\textbf{ID} &
\textbf{Polynomial in \(k[x,y,z,w]\)} &
\textbf{Use in the graphs} \\
\hline\hline

Noise-A &
\(w^2+1\) &
Assigned to unlabeled “filler” edges in \textit{Bowtie} and \textit{Plus}. Often becomes a unit after localization. \\
\hline

Noise-B &
\(x^2+y^2-5\) &
Assigned to secondary filler edges in \textit{Bowtie}; used to force a curved vanishing set. \\
\hline

\end{tabularx}
\caption{\textbf{Noise polynomials.} These are used on edges that are not part of the controlling cover data \((a,b,c)\); they are intended to verify that the predicted behavior is stable under additional irrelevant constraints.}
\label{tab:noise-polynomials}
\end{table}

\begin{table}[ht]
\centering
\footnotesize
\setlength{\extrarowheight}{0.25em}
\begin{tabularx}{\linewidth}{|l|Y|Y|Y|}
\hline
\textbf{ID} &
\textbf{Elements / labels} &
\textbf{Open set} &
\textbf{Graphs} \\
\hline\hline

PD2MonoLoc-xy &
\(a=x,\; b=1-xy,\; c=\langle z,\,w\rangle\). &
On \(U_{ab}\), the \(a\)- and \(b\)-edges trivialize and the remaining \(c\)-edges form a connected subgraph. &
GrilledCheese,\; SquareDiag,\; LinkedHexagons \\
\hline

PD2MonoLoc-yz &
\(a=y,\; b=1-yz,\; c=\langle x,\,w\rangle\). &
On \(U_{ab}\), the surviving \(c\)-edges are connected. &
GrilledCheese,\; SquareDiag,\; LinkedHexagons \\
\hline

PD2MonoLoc-xz &
\(a=x,\; b=1-xz,\; c=\langle y,\,w\rangle\). &
On \(U_{ab}\), the surviving \(c\)-edges are connected. &
GrilledCheese,\; SquareDiag,\; LinkedHexagons \\
\hline
\end{tabularx}
\caption{\textbf{Local connected \(2\)-mono scenarios.} Here \(V(a)\) and \(V(b)\) are forced to be disjoint, and \(c\) is a two-generated ideal. In these runs, the graph is arranged so that after restricting to \(U_{ab}\) all \(c\)-labeled edges remain connected.}
\label{tab:pd2monoloc-labels}
\end{table}

\begin{table}[ht]
\centering
\label{tab:controlled-results}
\begin{tabular}{lllll}
\toprule
Scenario & Graph & Rank & isFree? & Result \\
\midrule
PairCover-1      & Bowtie                        & 5  & true  & PASS \\
PairCover-2      & Bowtie                        & 5  & true  & PASS \\
PairCover-3      & Bowtie                        & 5  & true  & PASS \\
PairCover-4      & Bowtie                        & 5  & true  & PASS \\
\midrule
TripleCover-T3-1 & GrilledCheese                 & 6  & true  & PASS \\
TripleCover-T3-2 & GrilledCheese                 & 6  & true  & PASS \\
TripleCover-T3-3 & GrilledCheese                 & 6  & true  & PASS \\
TripleCover-T3-4 & GrilledCheese                 & 6  & true  & PASS \\
\midrule
TripleCover-PD3-1 & Plus                         & 12 & true  & PASS \\
TripleCover-PD3-1 & SquareDiag                   & 4  & true  & PASS \\
TripleCover-PD3-1 & LinkedHex (perimeter only)   & 13 & true  & PASS \\
TripleCover-PD3-1 & LinkedHex (center spokes)    & 13 & true  & PASS \\
TripleCover-PD3-2 & Plus                         & 12 & true  & PASS \\
TripleCover-PD3-2 & SquareDiag                   & 4  & true  & PASS \\
TripleCover-PD3-2 & LinkedHex (perimeter only)   & 13 & true  & PASS \\
TripleCover-PD3-2 & LinkedHex (center spokes)    & 13 & true  & PASS \\
TripleCover-PD3-3 & Plus                         & 12 & true  & PASS \\
TripleCover-PD3-3 & SquareDiag                   & 4  & true  & PASS \\
TripleCover-PD3-3 & LinkedHex (perimeter only)   & 13 & true  & PASS \\
TripleCover-PD3-3 & LinkedHex (center spokes)    & 13 & true  & PASS \\
TripleCover-PD3-4 & Plus                         & 12 & true  & PASS \\
TripleCover-PD3-4 & SquareDiag                   & 4  & true  & PASS \\
TripleCover-PD3-4 & LinkedHex (perimeter only)   & 13 & true  & PASS \\
TripleCover-PD3-4 & LinkedHex (center spokes)    & 13 & true  & PASS \\
\bottomrule
\end{tabular}
\caption{Computed global spline modules for the controlled scenarios.
In these families (\textit{PairCover-\#}, \textit{TripleCover-T3-\#}, and \textit{TripleCover-PD3-\#}) the covering hypotheses predict that the global spline module is free. Macaulay2 confirms this prediction for every tested graph, including both linked-hexagon variants (\emph{perimeter only} and \emph{center spokes}).}
\end{table}

\begin{table}[ht]
\centering
\label{tab:stress-results}
\begin{tabular}{lllll}
\toprule
Scenario & Graph & Rank & isFree? & Result \\
\midrule
StressTest-1 & Bowtie                      & 5  & true  & OBSERVED \\
StressTest-1 & GrilledCheese               & 6  & true  & OBSERVED \\
StressTest-1 & Plus                        & 12 & false & OBSERVED \\
StressTest-1 & SquareDiag                  & 4  & true  & OBSERVED \\
StressTest-1 & LinkedHex (perimeter only)  & 13 & false & OBSERVED \\
StressTest-1 & LinkedHex (center spokes)   & 13 & false & OBSERVED \\
\midrule
StressTest-2 & Bowtie                      & 5  & true  & OBSERVED \\
StressTest-2 & GrilledCheese               & 6  & true  & OBSERVED \\
StressTest-2 & Plus                        & 12 & false & OBSERVED \\
StressTest-2 & SquareDiag                  & 4  & true  & OBSERVED \\
StressTest-2 & LinkedHex (perimeter only)  & 13 & false & OBSERVED \\
StressTest-2 & LinkedHex (center spokes)   & 13 & false & OBSERVED \\
\midrule
StressTest-3 & Bowtie                      & 5  & true  & OBSERVED \\
StressTest-3 & GrilledCheese               & 6  & true  & OBSERVED \\
StressTest-3 & Plus                        & 12 & false & OBSERVED \\
StressTest-3 & SquareDiag                  & 4  & true  & OBSERVED \\
StressTest-3 & LinkedHex (perimeter only)  & 13 & false & OBSERVED \\
StressTest-3 & LinkedHex (center spokes)   & 13 & false & OBSERVED \\
\bottomrule
\end{tabular}
\caption{Computed global spline modules for the stress tests.
These inputs (\textit{StressTest-\#}) intentionally violate the covering hypotheses, so there is \emph{no theoretical prediction}. We record only the observed global freeness and rank. The two linked-hexagon geometries behave the same way numerically (rank \(13\), not free).}
\end{table}

\begin{table}[ht]
\centering
\label{tab:pd2monoloc-results}
\begin{tabular}{lllll}
\toprule
Scenario & Graph & Rank & isFree? & Result \\
\midrule
PD2MonoLoc-xy & GrilledCheese                & 6  & false & PASS \\
PD2MonoLoc-xy & SquareDiag                   & 4  & false & PASS \\
PD2MonoLoc-xy & LinkedHex (perimeter only)   & 13 & false & PASS \\
PD2MonoLoc-xy & LinkedHex (center spokes)    & 13 & false & PASS \\
\midrule
PD2MonoLoc-yz & GrilledCheese                & 6  & false & PASS \\
PD2MonoLoc-yz & SquareDiag                   & 4  & false & PASS \\
PD2MonoLoc-yz & LinkedHex (perimeter only)   & 13 & false & PASS \\
PD2MonoLoc-yz & LinkedHex (center spokes)    & 13 & false & PASS \\
\midrule
PD2MonoLoc-xz & GrilledCheese                & 6  & false & PASS \\
PD2MonoLoc-xz & SquareDiag                   & 4  & false & PASS \\
PD2MonoLoc-xz & LinkedHex (perimeter only)   & 13 & false & PASS \\
PD2MonoLoc-xz & LinkedHex (center spokes)    & 13 & false & PASS \\
\bottomrule
\end{tabular}
\caption{Computed global spline modules for the local connected \(2\)-mono scenarios.
Here the hypothesis predicts that the global spline module is \emph{not} free (the cover forces \(a\) and \(b\) to trivialize on \(U_{ab}\), leaving a connected \(2\)-generated \(c\)-locus). All twelve runs (\textit{PD2MonoLoc-xy}, \textit{-yz}, \textit{-xz} across all graphs tested) agree with the prediction: rank is unchanged, and \texttt{isFree} is \texttt{false}.}
\end{table}

\end{document}